\documentclass[12pt]{amsart}
\usepackage[centertags]{amsmath}
\usepackage{amssymb,amscd,amsfonts,latexsym,a4wide,amsthm}

\advance\headheight by 3pt

\renewcommand{\subjclass}[1]{\thanks{\emph{2010 Mathematics Subject Classification:}~#1}}
\renewcommand{\keywords}[1]{\thanks{\emph{Keywords and Phrases:}~#1}}
\renewcommand{\date}{\thanks{\today}}

\newtheorem{theorem}{Theorem}[section]
\newtheorem{corollary}[theorem]{Corollary}
\newtheorem{lemma}[theorem]{Lemma}
\newtheorem{proposition}[theorem]{Proposition}

\theoremstyle{definition}

\theoremstyle{remark}

\numberwithin{equation}{section}

\newcommand{\xv}{{\bf x}}

\newcommand{\yv}{{\bf y}}
\newcommand{\Yv}{{\bf Y}}
\newcommand{\zv}{{\bf z}}

\newcommand{\av}{{\bf a}}
\newcommand{\bv}{{\bf b}}

\newcommand{\uv}{{\bf u}}

\newcommand{\vv}{{\bf v}}

\newcommand{\nullv}{{\bf 0}}
\newcommand{\kv}{{\bf k}}

\renewcommand{\AA}{\mathcal{A}}

\newcommand{\NN}{\mathcal{N}}
\newcommand{\GG}{\mathcal{G}}
\newcommand{\FF}{\mathcal{F}}

\newcommand{\VV}{\mathcal{V}}

\newcommand{\PP}{\mathcal{P}}
\newcommand{\HH}{\mathcal{H}}
\newcommand{\WW}{\mathcal{W}}
\renewcommand{\SS}{\mathcal{S}}

\newcommand{\ve}{\varepsilon}

\newcommand{\fp}{\mathfrak{p}}
\newcommand{\fa}{\mathfrak{a}}

\newcommand{\kdots}{,\ldots ,}

\newcommand{\Qq}{\mathbb{Q}}
\newcommand{\OQq}{\overline{\mathbb{Q}}}
\newcommand{\Zz}{\mathbb{Z}}

\newcommand{\GL}{{\rm GL}}
\newcommand{\half}{\mbox{$\textstyle{\frac{1}{2}}$}}

\newcommand{\textfrac}[2]{\mbox{$\textstyle{\frac{#1}{#2}}$}}

\renewcommand{\gcd}{{\rm gcd}}

\newcommand{\expr}{\exp O(r)}
\newcommand{\Deg}{\overline{{\rm deg}}\,}
\newcommand{\Height}{\overline{h}}

\newcommand\irS{\varphi}

\title[Unit equations over finitely generated domains]
{Effective results for unit equations over finitely generated domains}
\date

\author[J.-H. Evertse]{Jan-Hendrik Evertse}
\address{J.-H. Evertse \newline
         \indent Universiteit Leiden, Mathematisch Instituut, \newline
         \indent Postbus 9512, 2300 RA Leiden, The Netherlands}
\email{evertse\char'100math.leidenuniv.nl}

\author[K. Gy\H{o}ry]{K\'{a}lm\'{a}n Gy\H{o}ry}
\address{K. Gy\H{o}ry \newline
         \indent Institute of Mathematics, University of Debrecen \newline
         \indent Number Theory Research Group, Hungarian Academy of Sciences and \newline
         \indent University of Debrecen \newline
         \indent H-4010 Debrecen, P.O. Box 12, Hungary}
\email{gyory\char'100science.unideb.hu}

\begin{document}

\maketitle

\begin{abstract}
Let $A\supset \Zz$ be a commutative domain which is finitely
generated over $\Zz$ as a $\Zz$-algebra and let $a,b,c$ be non-zero elements of $A$.
Extending earlier work of Siegel \cite[1921]{Sie21}, 
Mahler \cite[1933]{Mah33} and Parry \cite[1950]{Par50},
Lang \cite[1960]{Lang60} proved
that the equation (*) $a\ve +b\eta =c$ in $\ve ,\eta\in A^*$ 
has only finitely many solutions.
Using Baker's theory of logarithmic forms,
Gy\H{o}ry \cite[1974]{Gy74}, \cite[1979]{Gy79} proved that the solutions
of (*) can be determined effectively if $A$ is contained in an algebraic
number field.
In this paper we prove, in a precise quantitative form,
an effective finiteness result for equations (*)
over an arbitrary domain $A$ of characteristic $0$
which is finitely generated over $\Zz$.
Our main tools are already existing effective finiteness results
for (*) over number fields and function fields,
an effective specialization argument of Gy\H{o}ry 
\cite[1983]{Gy83}, \cite[1984]{Gy84}, 
and effective results of Seidenberg \cite[1974]{Sei74}
and Aschenbrenner \cite[2004]{Asc04} on linear equations 
over polynomial rings. 
We prove also an effective result for the
exponential equation 
$a\gamma_1^{v_1}\cdots \gamma_s^{v_s}+b\gamma_1^{w_1}\cdots \gamma_s^{w_s}=c$
in integers $v_1\kdots w_s$, where $a,b,c$ and $\gamma_1\kdots\gamma_s$
are non-zero elements of $A$.
\end{abstract}   

\section{Introduction}\label{1}

Let $A\supset\Zz$ be a commutative domain which is finitely generated over $\Zz$
as a $\Zz$-algebra.
As usual, we denote by $A^*$ the unit group of $A$.
We consider equations
\begin{equation}\label{1.1}
a\ve +b\eta =c\ \ \mbox{in } \ve ,\eta\in A^*
\end{equation}
where $a,b,c$ are non-zero elements of $A$. 
Such equations, usually called \emph{unit equations}, have a great number of applications.
For instance, the ring of $S$-integers in 
an algebraic number field is finitely generated over $\Zz$, so 
the $S$-unit equation in two unknowns is a special case of \eqref{1.1}.
In this paper, we consider equations \eqref{1.1} in the general case,
where $A$ may contain transcendental elements, too.

Siegel \cite[1921]{Sie21} proved that \eqref{1.1} has only finitely many solutions in the case that 
$A$ is the ring of integers of a number field,
and Mahler \cite[1933]{Mah33} did this in the case that $A =\Zz [1/p_1\cdots p_t]$ for certain primes
$p_1\kdots p_t$. 
For $S$-unit equations over number fields, the finiteness of the number of solutions of \eqref{1.1}
follows from work of Parry \cite[1950]{Par50}. 
Finally, Lang \cite[1960]{Lang60} proved for arbitrary finitely generated
domains $A$ that \eqref{1.1} has only finitely many solutions. 
The proofs of all these results are ineffective.

Baker \cite[1968]{Bak68} and Coates \cite[1968/69]{Coa68/69} implicitly proved effective finiteness
results for certain special ($S$-)unit equations. 
Later, Gy\H{o}ry \cite[1974]{Gy74}, \cite[1979]{Gy79} showed, 
in the case that $A$ is the ring of $S$-integers in a number field,
that the solutions of \eqref{1.1} can be determined effectively in principle.
His proof is based on estimates for linear forms in ordinary and $p$-adic logarithms of algebraic numbers.
In his papers \cite[1983]{Gy83} and \cite[1984]{Gy84}, Gy\H{o}ry introduced an
effective specialization argument, and he used this to establish effective finiteness results
for decomposable form equations and discriminant equations over a wide class
of finitely generated domains $A$ containing both algebraic and transcendental elements,
of which the elements have some ``good" effective representations. 
His results contain as a special case an effective finiteness result for equations \eqref{1.1}
over these domains. Gy\H{o}ry's method of proof could not be extended to arbitrary
finitely generated domains $A$.

It is the purpose of this paper
to prove an effective finiteness result for \eqref{1.1} 
over arbitrary finitely generated domains $A$. In fact, we give a quantitative
statement, with effective upper bounds for the ``sizes" of the solutions 
$\ve, \eta$. The main new ingredient of our proof is an effective result by 
Aschenbrenner \cite[2004]{Asc04} on systems of linear equations over polynomial
rings over $\Zz$.

We introduce the notation used in our theorems. Let again $A\supset \Zz$ be a
commutative domain which is finitely generated over $\Zz$,
say $A=\Zz [z_1\kdots z_r]$. Let $I$ be the ideal of polynomials
$f\in\Zz [X_1\kdots X_r]$ such that $f(z_1\kdots z_r)=0$. Then $I$ is finitely generated, hence
\begin{equation}\label{1.2}
A\cong \Zz [X_1\kdots X_r]/I,\ \ I=(f_1\kdots f_m)
\end{equation}
for some finite set of polynomials $f_1\kdots f_m\in\Zz [X_1\kdots X_r]$. 
We observe here that given $f_1\kdots f_m$,
it can be checked effectively whether $A$ is a domain
containing $\Zz$. Indeed, this holds if and only if  
$I$ is a prime ideal of $\Zz [X_1\kdots X_r]$ with $I\cap\Zz =(0)$, and the latter 
can be checked effectively for instance using Aschenbrenner \cite[Prop. 4.10, Cor. 3.5]{Asc04}.

Denote by $K$ the quotient field of $A$. For $\alpha \in A$, we call $f$ a \emph{representative} 
for $\alpha$, or say that $f$ represents $\alpha$ 
if $f\in\Zz [X_1\kdots X_r]$ and $\alpha =f(z_1\kdots z_r)$.
Further, for $\alpha\in K$, we call $(f,g)$ a \emph{pair of representatives} for $\alpha$
or say that $(f,g)$ represents $\alpha$ 
if $f,g\in\Zz [X_1\kdots X_r]$, $g\not\in I$ and
$\alpha =f(z_1\kdots z_r)/g(z_1\kdots z_r)$.
We say that $\alpha\in A$ (resp. $\alpha\in K$) is given if a representative
(resp. pair of representatives) for $\alpha$ is given.

To do explicit computations in $A$ and $K$, one needs
an \emph{ideal membership algorithm}
for $\Zz [X_1\kdots X_r]$, that is an algorithm which decides for any given
polynomial and ideal of $\Zz [X_1\kdots X_r]$ whether the polynomial belongs to the ideal.
In the literature there are various such algorithms; we mention only the algorithm
of Simmons \cite[1970]{Sim70}, and the more precise algorithm of Aschenbrenner \cite[2004]{Asc04}
which plays an important role in our paper; see Lemma \ref{le:2.5} below
for a statement of his result.   
One can perform arithmetic operations on $A$ and $K$ by 
using representatives. 
Further, one can decide effectively whether
two polynomials $f_1,f_2$ represent the same element of $A$, i.e., $f_1-f_2\in I$, 
or whether two pairs of polynomials $(f_1,g_1),(f_2,g_2)$ represent the same
element of $K$, i.e., $f_1g_2-f_2g_1\in I$, by using one of the ideal membership algorithms
mentioned above.  

The \emph{degree} $\deg f$ of a polynomial $f\in\Zz [X_1\kdots X_r]$ is by definition
its total degree. By the \emph{logarithmic height} $h(f)$ of $f$ we mean
the logarithm of the maximum of the absolute values of its coefficients.
The \emph{size} of $f$ is defined by $s(f):=\max (1,\deg f,h(f))$.
Clearly, there are only finitely many polynomials in $\Zz [X_1\kdots X_r]$ of size below
a given bound, and these can be determined effectively.

\begin{theorem}\label{th:1.1}
Assume that $r\geq 1$. 
Let $\widetilde{a},\widetilde{b},\widetilde{c}$ be representatives for $a,b,c$, respectively. 
Assume that
$f_1\kdots f_m$ and $\widetilde{a},\widetilde{b},\widetilde{c}$ all have degree at most $d$
and logarithmic height at most $h$, where $d\geq 1$, $h\geq 1$.
Then for each solution $(\ve ,\eta )$ of \eqref{1.1}, there are representatives
$\widetilde{\ve},\widetilde{\ve}',\widetilde{\eta},\widetilde{\eta}'$ of 
$\ve,\ve^{-1},\eta,\eta^{-1}$,
respectively, such that
\[
s(\widetilde{\ve}),\, s(\widetilde{\ve}'),\, s(\widetilde{\eta}),\, s(\widetilde{\eta}')
\leq \exp \Big( (2d)^{c_1^r}(h+1)\Big),
\]
where $c_1$ is an effectively computable absolute constant $>1$.
\end{theorem}

By a theorem of Roquette \cite[1958]{Roq58}, 
the unit group of a domain finitely generated over $\Zz$
is finitely generated. In the case that $A=O_S$ is the ring of $S$-integers of a number field
it is possible to determine effectively a system of generators for $A^*$, and this was used by
Gy\H{o}ry in his effective finiteness proof for \eqref{1.1} with $A=O_S$.
However, no general algorithm is known to determine a system of generators for the unit group
of an arbitrary finitely generated domain $A$.
In our proof of Theorem \ref{th:1.1}, we do not need any information on the generators of $A^*$.

By combining Theorem \ref{th:1.1} with an ideal membership algorithm for \\
$\Zz [X_1\kdots X_r]$,
one easily deduces the following:

\begin{corollary}\label{th:1.2}
Given $f_1\kdots f_m, a,b,c$, the solutions of \eqref{1.1} can be determined effectively.
\end{corollary}

\begin{proof}
Clearly, $\ve ,\eta$ is a solution of \eqref{1.1} if and only if there 
are polynomials 
$\widetilde{\ve },\widetilde{\ve}',\widetilde{\eta},\widetilde{\eta}'\in\Zz [X_1\kdots X_r]$ 
such that $\widetilde{\ve},\widetilde{\eta}$ represent $\ve ,\eta$, and
\begin{equation}\label{1.3}
\widetilde{a}\cdot \widetilde{\ve}+\widetilde{b}\cdot \widetilde{\eta}-\widetilde{c},\ \,
\widetilde{\ve}\cdot \widetilde{\ve}' -1,\ \, \widetilde{\eta}\cdot\widetilde{\eta}' -1\in I.
\end{equation}
Thus, we obtain all solutions of  \eqref{1.1}
by checking, for each quadruple of polynomials 
$\widetilde{\ve },\widetilde{\ve}',\widetilde{\eta},\widetilde{\eta}'\in\Zz [X_1\kdots X_r]$ 
of size at most $\exp \Big( (2d)^{c_1^r}(h+1)\Big)$ whether it satisfies \eqref{1.3}.
Further, using the ideal membership algorithm,
it can be checked effectively whether two different pairs 
$(\widetilde{\ve},\widetilde{\eta})$ represent the same solution of \eqref{1.1}.
Thus, we can make a list of representatives, one for each solution of \eqref{1.1}.
\end{proof}

Let $\gamma_1\kdots\gamma_s$ be multiplicatively independent elements of $K^*$
(the multiplicative independence of $\gamma_1\kdots\gamma_s$ can be checked effectively for instance
using Lemma \ref{le:7.2} below). Let again $a,b,c$ be non-zero elements of $A$ and consider the equation
\begin{equation}\label{1.4}
a\gamma_1^{v_1}\cdots\gamma_s^{v_s}+b\gamma_1^{w_1}\cdots\gamma_s^{w_s}=c\ \ \mbox{in } 
v_1\kdots v_s,\, w_1\kdots w_s\in\Zz .
\end{equation}

\begin{theorem}\label{th:1.3}
Let $\widetilde{a},\widetilde{b},\widetilde{c}$
be representatives for $a,b,c$ and for $i=1\kdots s$, let $(g_{i1},g_{i2})$ be a pair
of representatives for $\gamma_i$.
Suppose that
$f_1\kdots f_m$, $\widetilde{a},\widetilde{b},\widetilde{c}$, and
$g_{i1},g_{i2}$ ($i=1\kdots s$)
all have degree at most $d$ and logarithmic height at most $h$, where $d\geq 1$, $h\geq 1$.
Then for each solution $(v_1\kdots w_s)$ of \eqref{1.4} we have
\[
\max \big(|v_1|\kdots |v_s|,\, |w_1|\kdots |w_s|\big)\leq \exp \Big( (2d)^{c_2^{r+s}}(h+1)\Big),
\]
where $c_2$ is an effectively computable absolute constant $>1$.
\end{theorem}

\noindent
An immediate consequence of Theorem \ref{th:1.3} is that for given $f_1\kdots f_m$,
$a,b,c$ and $\gamma_1\kdots \gamma_s$, the solutions of \eqref{1.4} can be determined
effectively. 

Since every domain finitely generated over $\Zz$ has a finitely generated   
unit group, equation \eqref{1.1} maybe viewed as a special case of \eqref{1.4}. 
But since no general effective algorithm is known to find a finite system of generators
for the unit group of a finitely generated domain, 
we cannot deduce an effective 
result for \eqref{1.1} from Theorem \ref{th:1.3}. 
In fact, we argue reversely, and prove Theorem \ref{th:1.3} by combining 
Theorem \ref{th:1.1} with an effective result
on Diophantine equations of the type $\gamma_1^{v_1}\cdots\gamma_s^{v_s}=\gamma_0$ 
in integers $v_1\kdots v_s$,
where $\gamma_1\kdots\gamma_s,\gamma_0\in K^*$ (see Corollary \ref{co:7.3} below).

The idea of the proof of Theorem \ref{th:1.1} is roughly as follows.
We first estimate the degrees of the representatives of $\ve ,\eta$ 
using Mason's effective result \cite[1983]{Mas83} 
on two term $S$-unit equations over function fields.
Next, we apply many different specialization maps $A\to \OQq$ to \eqref{1.1}
and obtain in this manner a large system of $S$-unit equations over number fields.
By applying an existing effective finiteness result for such $S$-unit equations
(e.g., Gy\H{o}ry and Yu \cite[2006]{GyYu06})
we collect enough information to retrieve an effective upper bound
for the heights of the representatives of $\ve ,\eta$.
In our proof, we apply the specialization maps on a domain $B\supset A$
of a special type which can be dealt with more easily.
In the construction of $B$, we use
an effective result of Seidenberg \cite[1974]{Sei74} 
on systems of linear equations over polynomial rings over arbitrary fields.
To be able to go back to equation \eqref{1.1} over $A$,
we need an effective procedure to decide whether a given element of $B$ belongs to $A^*$.
For this decision procedure, we apply an effective result of Aschenbrenner 
\cite[2004]{Asc04} on systems of linear
equations over polynomial rings over $\Zz$.

The above approach was already followed by Gy\H{o}ry \cite[1983]{Gy83}, \cite[1984]{Gy84}.
However, in these papers the domains $A$ are represented over $\Zz$ in a different way.
Hence, to select those solutions from $B$ of the equations under consideration
which belong to $A$, certain restrictions on the domains $A$ had to be imposed.

In a forthcoming paper, we will give some applications of our above theorems
and our method of proof to other classes of Diophantine equations over finitely generated
domains. 

\section{Effective linear algebra over polynomial rings}\label{2}

We have collected some effective results for systems of linear equations
to be solved in polynomials with coefficients in a field,
or with coefficients in $\Zz$. 

Here and in the remainder of this paper, we write 
\[
\log^* x:=\max (1,\log x)\ \mbox{for $x>0$,  } \log^* 0 :=1.
\]
We use notation $O(\cdot )$ as an abbreviation for $c\times$ the expression
between the parentheses, where $c$ is an effectively computable
absolute constant. At each occurrence of $O(\cdot )$,
the value of $c$ may be different.

Given a commutative domain $R$, we denote by $R^{m,n}$ the $R$-module of 
$m\times n$-matrices with entries in $R$
and by $R^n$ the $R$-module of $n$-dimensional column vectors with entries in $R$.
Further, $\GL_n(R)$ denotes the group of matrices in $R^{n,n}$ 
with determinant in the unit group $R^*$.
The degree of a polynomial $f\in R[X_1\kdots X_N]$,
that is, its total degree, is denoted by $\deg f$.

From matrices $A,B$ with the same number of rows, 
we form a matrix $[A,B]$ by placing the columns of $B$
after those of $A$.
Likewise, from two matrices $A,B$ with the same number of columns 
we form $\left[\begin{smallmatrix}A\\B\end{smallmatrix}\right]$ 
by placing the rows of $B$ below those of $A$.

The logarithmic height $h(S)$ of a finite set $S=\{ a_1\kdots a_t\}\subset\Zz$
is defined by $h(S):=\log\max (|a_1|\kdots |a_t|)$.
The logarithmic height $h(U)$ of a matrix with entries in $\Zz$
is defined by the logarithmic height of the set of entries of $U$.
The logarithmic height $h(f)$ of a polynomial with coefficients in $\Zz$
is the logarithmic height of the set of coefficients of $f$.

\begin{lemma}\label{le:2.1}
Let $U\in\Zz^{m,n}$. Then
the $\Qq$-vector space of $\yv\in\Qq^n$ with $U\yv =\nullv$ 
is generated by vectors in $\Zz^n$
of logarithmic height at most $mh(U)+\half m\log m$.
\end{lemma}

\begin{proof}
Without loss of generality we may assume that $U$ has rank $m$,
and moreover, that the matrix $B$
consisting of the first $m$ columns of $U$ is invertible. 
Let $\Delta := \det B$.
By multiplying with $\Delta B^{-1}$, we can rewrite $U\yv =\nullv $ 
as $[\Delta I_m ,\, C]\yv =\nullv$,
where $I_m$ is the $m\times m$-unit matrix,
and $C$ consists of $m\times m$-subdeterminants of $U$.
The solution space of this system is generated by the columns of 
$\left[\begin{smallmatrix}-C\\ \Delta I_{n-m}\end{smallmatrix}\right]$.
An application of Hadamard's inequality gives the upper bound from the lemma
for the logarithmic heights of these columns.
\end{proof}

\begin{proposition}\label{le:2.2}
Let $F$ be a field, $N\geq 1$, and $R:=F [X_1\kdots X_N]$. 
Further, let $A$ be an $m\times n$-matrix and $\bv$ and $m$-dimensional
column vector, both consisting of polynomials from $R$
of degree $\leq d$ where $d\geq 1$.
\\[0.15cm]
(i) The $R$-module of $\xv\in R^n$ with $A\xv =\nullv$
is generated by vectors $\xv$ whose coordinates are polynomials
of degree at most $(2md)^{2^N}$.
\\[0.15cm]
(ii) Suppose that $A\xv =\bv$ is solvable in $\xv\in R^n$.
Then it has a solution $\xv$ whose coordinates are polynomials
of degree at most $(2md)^{2^N}$.
\end{proposition}

\begin{proof}
See Aschenbrenner \cite[Thms. 3.2, 3.4]{Asc04}.
Results of this type were obtained earlier, but not with a completely correct proof,
by Hermann \cite[1926]{Her26} and Seidenberg \cite[1974]{Sei74}.
\end{proof}

\begin{corollary}\label{co:2.3}
Let $R:= \Qq [X_1\kdots X_N]$. Further,
Let $A$ be an $m\times n$-matrix of polynomials in $\Zz [X_1\kdots X_N]$
of degrees at most $d$ and logarithmic heights at most $h$ where $d\geq 1$, $h\geq 1$.
Then the $R$-module of $\xv\in R^n$ with $A\xv =\nullv$ is generated
by vectors $\xv$, consisting of polynomials in 
$\Zz [X_1\kdots X_N]$ of degree at most $(2md)^{2^N}$ and height 
at most $(2md)^{6^N}(h+1)$.
\end{corollary}

\begin{proof}
By Proposition \ref{le:2.2} (i) we have to study $A\xv =\nullv$, restricted to 
vectors $\xv\in R^n$ consisting of polynomials of degree at most $(2d)^{2^N}$.
The set of these $\xv$ 
is a finite dimensional $\Qq$-vector space, and we have to prove
that it is generated by vectors whose coordinates 
are polynomials in $\Zz [X_1\kdots X_N]$
of logarithmic height at most $(2md)^{6^N}(h+1)$.
 
If $\xv$ consists of polynomials of degree at most $(2md)^{2^N}$,
then $A\xv$ consists of $m$ polynomials with coefficients in $\Qq$ of degrees
at most $(2md)^{2^N}+d$, all whose coefficients have to be set to $0$.
This leads to a system of linear equations
$U\yv =\nullv$, where $\yv$ consists of the coefficients of the polynomials
in $\xv$ and $U$ consists of integers of logarithmic heights at most $h$.
Notice that the number $m^*$ of rows of $U$ is  $m$ times the number of 
monomials in $N$ variables of degree at most $(2md)^{2^N}+d$, that is
\[
m^*\leq m\binom{(2md)^{2^N}+d+N}{N}.
\]
By Lemma \ref{le:2.1} the solution space of $U\yv =\nullv$ is generated by
integer vectors of logarithmic height at most 
\[
m^*h+\half m^*\log m^*\leq (2md)^{6^N}(h+1).
\]
This completes the proof of our corollary.
\end{proof}

\begin{lemma}\label{le:2.4}
Let $U\in\Zz^{m,n}$, $\bv\in\Zz^m$ be such that $U\yv=\bv$ 
is solvable in $\Zz^n$.
Then it has a solution
$\yv\in\Zz^n$ with $h(\yv )\leq mh([U,\bv ])+\half m\log m$.
\end{lemma}

\begin{proof}
Assume without loss of generality that $U$ and $[U,\bv ]$ have rank $m$.
By a result of Borosh, Flahive, Rubin and Treybig \cite[1989]{BFRT89},
$U\yv =\bv$ has a solution $\yv\in\Zz^n$ such that the absolute values of the entries of $\yv$
are bounded above by the maximum of the absolute values of the 
$m\times m$-subdeterminants of
$[U,\bv ]$. The upper bound for $h(\yv )$ as in the lemma 
easily follows from Hadamard's inequality.
\end{proof}

\begin{proposition}\label{le:2.5}
Let $N\geq 1$ and let $f_1\kdots f_m,b\in\Zz [X_1\kdots X_N]$
be polynomials of degrees at most $d$ and logarithmic heights at most $h$
where $d\geq 1$, $h\geq 1$,
such that
\begin{equation}\label{2.1}
f_1x_1+\cdots +f_mx_m=b
\end{equation}
is solvable in $x_1\kdots x_m\in\Zz [X_1\kdots x_N]$.
Then \eqref{2.1} has a solution in polynomials 
$x_1\kdots x_m\in\Zz [X_1\kdots X_N]$ with 
\begin{equation}\label{2.2}
\deg x_i \leq (2d)^{\exp O(N\log^* N)}(h+1),\ \ 
h(x_i)\leq (2d)^{\exp O(N\log^* N)}(h+1)^{N+1}
\end{equation}
for $i=1\kdots m$. 
\end{proposition}

\begin{proof}
Aschenbrenner's main theorem \cite [Theorem A]{Asc04} states that 
Eq. \eqref{2.1} has a solution $x_1\kdots x_m\in\Zz [X_1\kdots X_N]$
with $\deg x_i\leq d_0$ for $i=1\kdots m$, where 
\[
d_0=(2d)^{\exp O(N\log^*N)}(h+1).
\]
So it remains to show the existence of a solution with small logarithmic height.

Let us restrict to solutions $(x_1\kdots x_m)$ of \eqref{2.1}
of degree $\leq d_0$, 
and denote by $\yv$ the vector of coefficients 
of the polynomials $x_1\kdots x_m$. 
Then \eqref{2.1} translates into a system of linear equations  $U\yv =\bv$
which is solvable over $\Zz$.
Here, the number of equations, i.e., number of rows of $U$,
is equal to $m^*:= \binom{d_0+d+N}{N}$. Further, $h(U, \bv )\leq h$.
By Lemma \ref{le:2.4}, $U\yv =\bv$ has a solution $\yv$ 
with coordinates in $\Zz$ of height at most 
\[
m^*h +\half m^*\log m^*\leq (2d)^{\exp O(N\log^* N)}(h+1)^{N+1}.
\]
It follows that \eqref{2.1} has a solution
$x_1\kdots x_m\in\Zz [X_1\kdots X_N]$ satisfying \eqref{2.2}.
\end{proof} 

\noindent
{\bf Remarks. 1.}
Aschenbrenner gives in \cite{Asc04} an example which shows that the upper bound for the
degrees of the $x_i$ cannot depend on $d$ and $N$ only.
\\[0.15cm]
{\bf 2.} The above lemma gives an effective criterion for ideal membership
in $\Zz [X_1\kdots X_N]$. 
Let $b\in\Zz [X_1\kdots X_N]$ be given. Further, suppose that an ideal
$I$ of $\Zz [X_1\kdots X_N]$ is given by a finite set of generators
$f_1\kdots f_m$. 
By the above lemma, if $b\in I$ then
there are $x_1\kdots x_m\in\Zz [X_1\kdots X_N]$
with upper bounds for the degrees and heights as in \eqref{2.2}
such that $b=\sum_{i=1}^m x_if_i$. It requires only a finite computation
to check whether such $x_i$ exist.

\section{A reduction}\label{3}

We reduce the general unit equation \eqref{1.1} to a unit equation over a domain $B$ of a special type which can be dealt with more easily.

Let again $A=\Zz [z_1\kdots z_r]\supset\Zz$ be a commutative domain finitely
generated over $\Zz$ and denote by $K$ the quotient field of $A$. We assume that $r>0$.
We have
\begin{equation}\label{3.-1}
A\cong \Zz [X_1\kdots X_r]/I
\end{equation}
where $I$ is the ideal of polynomials $f\in\Zz [X_1\kdots X_r]$
such that $f(z_1\kdots z_r)\\=0$. The ideal $I$ is finitely generated.
Let $d\geq 1$, $h\geq 1$ and assume that  
\begin{equation}\label{3.0}
I=(f_1\kdots f_m)\ \ \mbox{with } \deg f_i\leq d,\ \ h(f_i)\leq h\ 
(i=1\kdots m).
\end{equation}

Suppose that $K$ has transcendence degree $q\geq 0$. 
In case that $q>0$, we assume without loss of generality
that $z_1\kdots z_q$ form a transcendence basis of $K/\Qq$. 
We write
$t:= r-q$ and rename $z_{q+1}\kdots z_r$ as $y_1\kdots y_t$,
respectively. In case that $t=0$ we have $A=\Zz [z_1\kdots z_q]$, $A^*=\{\pm 1\}$
and Theorem \ref{th:1.1} is trivial. So we assume henceforth that $t>0$. 

Define 
\begin{eqnarray*}
&&A_0:=\Zz [z_1\kdots z_q],\ \ K_0:=\Qq (z_1\kdots z_q)\ \ \mbox{if } q>0,
\\[0.15cm]
&&A_0:=\Zz,\ \ K_0:=\Qq\ \ \mbox{if } q=0.
\end{eqnarray*}
Then
\[
A=A_0[y_1\kdots y_t],\ \ K=K_0(y_1\kdots y_t).
\]
Clearly,
$K$ is a finite extension of $K_0$, so in particular an algebraic number field if $q=0$. 
Using standard algebra techniques,
one can show that there exist $y\in A$, $f\in A_0$ such that $K=K_0(y)$,  
$y$ is integral over $A_0$, and
\[ 
A\subseteq B:=A_0[f^{-1},y],\ \ \ \ a,b,c\in B^*.
\]
If $\ve,\eta\in A^*$ is a solution to \eqref{1.1}, 
then $\ve_1:= a\ve/c$, $\eta_1 :=b\eta/c$ satisfy
\begin{equation}\label{3.-2}
\ve_1 +\eta_1 =1,\ \ \ve_1 ,\eta_1\in B^*.
\end{equation} 
At the end of this section, we formulate Proposition \ref{pr:3.6} which gives an effective result
for equations of the type \eqref{3.-2}.
More precisely, we introduce an other type of degree and height $\Deg (\alpha )$ and $\Height (\alpha )$ for elements $\alpha$ of $B$, and give effective upper bounds for the 
$\Deg$ and $\Height$ 
of $\ve_1,\eta_1$. 
Subsequently we deduce Theorem \ref{th:1.1}.

The deduction of Theorem \ref{th:1.1} is based on some auxiliary results
which are proved first. 
We start with an explicit construction of $y,f$, with effective upper bounds
in terms of $r$, $d$, $h$ and $a,b,c$
for the degrees and logarithmic heights of $f$ and of the coefficients in $A_0$ 
of the monic minimal polynomial of $y$ over $A_0$. 
Here we follow more or less Seidenberg \cite[1974]{Sei74}.
Second, for a given solution $\ve ,\eta$ of \eqref{1.1},
we derive  effective upper bounds
for the degrees and logarithmic heights
of representatives for $\ve$, $\ve^{-1}$, $\eta$, $\eta^{-1}$
in terms of $\Deg (\ve_1 )$, $\Height (\ve_1 )$, $\Deg (\eta_1)$, $\Height (\eta_1)$.
Here we use Proposition \ref{le:2.5} (Aschenbrenner's result).

We introduce some further notation.
First let $q>0$. Then
since $z_1\kdots z_q$ are algebraically independent, we may view them as independent
variables,
and for $\alpha\in A_0$, we denote by $\deg\alpha$, $h(\alpha )$ 
the total degree and logarithmic height of $\alpha$, viewed as polynomial in $z_1\kdots z_q$.
In case that $q=0$, we have $A_0=\Zz$, and we agree that $\deg \alpha =0$, 
$h(\alpha )=\log |\alpha|$ for $\alpha\in A_0$.
We frequently use the following estimate, valid for all $q\geq 0$:

\begin{lemma}\label{le:3.0}
Let $g_1\kdots g_n\in A_0$ and $g=g_1\cdots g_n$. Then
\[
|h(g)-\sum_{i=1}^n h(g_i)|\leq q\deg g.
\]
\end{lemma}

\begin{proof}
See Bombieri and Gubler \cite[Lemma 1.6.11, pp. 27]{BoGu06}.
\end{proof}

We write
$\Yv =(X_{q+1}\kdots X_r)$ and $K_0(\Yv ):= K_0(X_{q+1}\kdots X_r)$, etc.
Given $f\in \Qq (X_1\kdots X_r)$ we denote by $f^*$ the rational function of
$K_0(\Yv )$ obtained by substituting $z_i$ for $X_i$ for $i=1\kdots q$ (and $f^*=f$ if $q=0$).
We view elements $f^*\in A_0[\Yv ]$ as polynomials in $\Yv$
with coefficients in $A_0$. We denote by $\deg_{\Yv} f^*$ the (total) degree of $f^*\in K_0[\Yv ]$
with respect to $\Yv$. We recall that $\deg g$ is defined for elements of $A_0$ and is
taken with respect to $z_1\kdots z_q$.  
With this notation, we can rewrite \eqref{3.-1}, \eqref{3.0} as
\begin{equation}\label{3.0a}
\left\{\begin{array}{l}
A\cong A_0[\Yv ]/(f_1^*\kdots f_m^*),
\\[0.1cm]
\deg_{\Yv} f_i^* \leq d\ \mbox{for } i=1\kdots m,
\\[0.1cm]
\mbox{the coefficients of $f_1^*\kdots f_m^*$ in $A_0$ have degrees at most $d$}
\\
\mbox{and logarithmic heights at most $h$.}
\end{array}\right.
\end{equation}
Put $D:= [K:K_0]$
and denote by $\sigma_1\kdots \sigma_D$ the $K_0$-
isomorphic embeddings of $K$ in an algebraic closure $\overline{K_0}$
of $K_0$.

\begin{lemma}\label{le:3.1}
(i) We have $D\leq d^t$.
\\
(ii) There exist integers $a_1\kdots a_t$ with $|a_i|\leq D^2$ for $i=1\kdots t$
such that for $w:=a_1y_1+\cdots +a_ty_t$ we have $K=K_0(w)$.
\end{lemma}

\begin{proof}
(i) The set 
\[
\WW := \{ \yv \in \overline{K_0}^t:\ f_1^*(\yv )=\cdots =f_m^*(\yv )=0\}
\]
consists precisely of the images of $(y_1\kdots y_t)$ under $\sigma_1\kdots\sigma_D$.
So we have to prove that $\WW$ has cardinality at most $d^t$.

In fact, this follows from a repeated application of B\'{e}zout's Theorem.
Given $g_1\kdots g_k\in K_0[\Yv ]$, we denote by $\VV(g_1\kdots g_k)$
the common set of zeros of $g_1\kdots g_k$ in $\overline{K_0}^t$.
Let $g_1:=f_1^*$.
Then by the version of B\'{e}zout's Theorem in Hartshorne 
\cite[p. 53, Thm. 7.7]{Har77},
the irreducible components of $\VV(g_1)$ have dimension $t-1$, 
and the sum of their degrees
is at most $\deg_{\Yv} g_1\leq d$. 
Take a $\overline{K_0}$-linear combination $g_2$ of $f_1^*\kdots f_m^*$
not vanishing identically on any of the irreducible components of $\VV (g_1)$.
For any of these components, say $\VV$, the intersection of $\VV$ and $\VV (g_2)$
is a union of irreducible
components, each of dimension $t-2$, whose degrees have sum at most 
$\deg_{\Yv} g_2\cdot\deg\VV\leq d\deg \VV$.
It follows that the irreducible components of $\VV (g_1,g_2)$ have dimension $t-2$
and that the sum of their degrees is at most $d^2$. 
Continuing like this, we see that
there are linear combinations $g_1\kdots g_t$ of $f_1^*\kdots f_m^*$ such that
for $i=1\kdots t$, the irreducible components of $\VV (g_1\kdots g_i)$ 
have dimension $d-i$
and the sum of their degrees is at most $d^i$. 
For $i=t$ it follows that $\VV (g_1\kdots g_t)$ is a set of at most $d^t$ points. 
Since $\WW\subseteq \VV(g_1\kdots g_t)$ this proves (i).

(ii) Let $a_1\kdots a_t$ be integers. Then $w:=\sum_{i=1}^t a_iy_i$
generates $K$ over $K_0$ if and only if
$\sum_{j=1}^t a_j\sigma_i(y_j)$ ($i=1\kdots D$) are distinct.
There are integers $a_i$ with $|a_i|\leq D^2$ for which this holds.
\end{proof}

\begin{lemma}\label{le:3.2}
There are $\GG_0\kdots \GG_D\in A_0$ such that
\begin{eqnarray}
\label{3.1}
&&\sum_{i=0}^D \GG_iw^{D-i}=0,\ \ \GG_0\GG_D\not= 0,
\\
\label{3.1a}
&&\deg\GG_i\leq (2d)^{\expr},\ \ \ h(\GG_i )\leq (2d)^{\expr}(h+1)\ \ \
(i=0\kdots D).
\end{eqnarray}
\end{lemma}

\begin{proof}
In what follows we write 
$\Yv =(X_{q+1}\kdots X_r)$ and
$\Yv^{\uv}:=X_{q+1}^{u_1}\cdots X_{q+t}^{u_t}$,
$|\uv |:= u_1+\cdots +u_t$ for tuples of non-negative integers 
$\uv =(u_1\kdots u_t)$. Further, we define $W:=\sum_{j=1}^t a_jX_{q+j}$.

$\GG_0\kdots \GG_D$ as in \eqref{3.1} clearly exist since $w$
has degree $D$ over $K_0$.
By \eqref{3.0a}, there are 
$g_1^*\kdots g_m^*\in A_0[\Yv ]$
such that
\begin{equation}\label{3.2}
\sum_{i=0}^D \GG_iW^{D-i}=
\sum_{j=1}^m g_j^*f_j^*.
\end{equation}
By Proposition \ref{le:2.2} (ii), applied with the field $F=K_0$, 
there are polynomials
$g_j^*\in K_0[\Yv ]$ (so with coefficients being rational
functions in $\zv$) satisfying \eqref{3.2} of degree at most 
$(2\max(d,D))^{2^t}\leq (2d^t)^{2^t}=: d_0$ in $\Yv$.
By multiplying $\GG_0\kdots\GG_D$ with an appropriate non-zero factor from 
$A_0$
we may assume that the $g_j^*$ are polynomials in $A_0[\Yv ]$ of degree at most $d_0$ in $\Yv$.
By considering \eqref{3.2} with such polynomials $g_j^*$, 
we obtain
\begin{equation}\label{3.3}
\sum_{i=0}^D \GG_iW^{D-i}=
\sum_{j=1}^m \Big(\sum_{|\uv |\leq d_0} g_{j,\uv}\Yv^{\uv}\Big)\cdot
\Big(\sum_{|\vv |\leq d} f_{j,\vv}\Yv^{\vv}\Big),
\end{equation}
where $g_{j,\uv}\in A_0$
and $f_j^*=\sum_{|\vv |\leq d} f_{j,\vv}\Yv^{\vv}$
with $f_{j,\vv}\in A_0$.
We view $\GG_0\kdots\GG_D$ and the polynomials $g_{j,\uv}$ 
as the unknowns of \eqref{3.3}.
Then \eqref{3.3} has solutions with $\GG_0\GG_D\not= 0$. 
 
We may view \eqref{3.3} as a system of linear equations 
$\AA\xv =\nullv$ over $K_0$, where
$\xv$ consists of 
$\GG_i$ ($i=0\kdots D$) and
$g_{j,\uv}$ ($j=1\kdots m$, $|\uv |\leq d_0$).
By Lemma \ref{le:3.1} and an elementary estimate,
the polynomial $W^{D-i}=(\sum_{k=1}^t a_kX_{q+k})^{D-i}$
has logarithmic height at most 
$O(D\log (2D^2t))\leq (2d)^{O(t)}$. By combining this with \eqref{3.0a}, 
it follows that the entries of the matrix $\AA$ are elements of $A_0$ 
of degrees at most $d$  
and logarithmic heights at most $h_0:=\max ((2d)^{O(t)},h)$.
Further, the number of rows of $\AA$ is at most the number of monomials
in $\Yv$ of degree at most $d_0+d$ which is bounded above by 
$m_0:=\binom{d_0+d+t}{t}$.
So by Corollary \ref{co:2.3}, the solution module of \eqref{3.3} is generated
by vectors $\xv =(\GG_0\kdots \GG_D,\,\{ g_{i,\uv}\})$, consisting
of elements from $A_0$ of degree and height at most
\[
\big(2m_0d\big)^{2^q}\leq (2d)^{\expr},\ \ 
\big(2m_0d\big)^{6^q}(h_0+1)\leq (2d)^{\expr}(h+1),
\]
respectively.

At least one of these vectors $\xv$ must have $\GG_0\GG_D\not= 0$ 
since otherwise \eqref{3.3} would have no solution with $\GG_0\GG_D\not= 0$, 
contradicting \eqref{3.1}.
Thus, there exists a solution $\xv$ whose components $\GG_0\kdots\GG_D$ 
satisfy both \eqref{3.1}, \eqref{3.1a}.
This proves our lemma.
\end{proof}

It will be more convenient to work with
\[
y:=\GG_0w=\GG_0\cdot (a_1y_1+\cdots +a_ty_t).
\]
In the case $D=1$ we set $y:=1$.
The following properties of $y$ follow at once from Lemmas \ref{le:3.0}--\ref{le:3.2}.

\begin{corollary}\label{co:3.3}
We have 
$K=K_0(y)$, $y\in A$, $y$ is integral over $A_0$, and $y$ has minimal polynomial
$\FF (X)=X^D+\FF_1X^{D-1}+\cdots +\FF_D$ over $K_0$ with
\[
\FF_i\in A_0,\ \ \deg \FF_i\leq (2d)^{\expr},
\ h(\FF_i)\leq (2d)^{\expr}(h+1)
\]
for $i=1\kdots D$.
\end{corollary}

Recall that $A_0=\Zz$ if $q=0$ and $\Zz [z_1\kdots z_q]$ if $q>0$, where in the latter case,
$z_1\kdots z_q$ are algebraically independent.
Hence $A_0$ is a unique factorization domain,
and so 
the gcd of a finite set of elements of $A_0$ is well-defined
and up to sign uniquely determined. With every element $\alpha\in K$ 
we can associate an up to sign unique tuple 
$P_{\alpha ,0}\kdots P_{\alpha ,D-1},Q_{\alpha}$ of elements of $A_0$ such that
\begin{equation}\label{3.4}
\alpha = Q_{\alpha}^{-1}\sum_{j=0}^{D-1} P_{\alpha ,j}y^j\ \ 
\mbox{with } 
Q_{\alpha}\not= 0,\ \gcd (P_{\alpha ,0}\kdots P_{\alpha ,D-1},Q_{\alpha})=1.
\end{equation}
Put
\begin{equation}\label{3.4b}
\left\{\begin{array}{l}
\Deg \alpha := \max (\deg P_{\alpha ,0}\kdots \deg P_{\alpha ,D-1}, 
\deg Q_{\alpha}),
\\[0.15cm]
\Height (\alpha ):= 
\max \big( h(P_{\alpha ,0})\kdots h(P_{\alpha ,D-1}),h(Q_{\alpha})\big)
\end{array}\right. .
\end{equation}
Then for $q=0$ we have $\Deg\alpha =0$, 
$\Height (\alpha )=\log\max \big(|P_{\alpha ,0}|\kdots |P_{\alpha ,D-1}|,|Q_{\alpha}|\big)$.

\begin{lemma}\label{le:3.4}
Let $\alpha\in K^*$ and let $(a,b)$ be a pair of representatives for $\alpha$,
with $a,b\in\Zz [X_1\kdots X_r]$, $b\not\in I$. 
Put $d^*:=\max (d,\deg a , \deg b)$, $h^*:=\max (h,h(a),h(b))$.
Then
\begin{equation}
\label{3.4.a}
\Deg\alpha \leq (2d^*)^{\expr},\ \ \Height (\alpha )\leq (2d^*)^{\expr} (h^*+1).
\end{equation}
\end{lemma}

\begin{proof}
Consider the linear equation
\begin{equation}\label{3.5}
Q\cdot \alpha =\sum_{j=0}^{D-1} P_jy^j
\end{equation}
in unknowns $P_0\kdots P_{D-1},Q\in A_0$. This equation has a solution
with $Q\not= 0$, since $\alpha\in K=K_0(y)$ and $y$ has degree $D$ over $K_0$.
Write again $\Yv =(X_{q+1}\kdots X_r)$ and put $Y:=\GG_0\cdot (\sum_{j=1}^t a_jX_{q+j})$.
Let $a^*,\, b^*\in A_0[\Yv ]$ be obtained from $a,b$ by substituting
$z_i$ for $X_i$ for $i=1\kdots q$ ($a^*=a$, $b^*=b$ if $q=0$). 
By \eqref{3.0a}, there are $g_j^*\in A_0[\Yv ]$ 
such that
\begin{equation}\label{3.6}
Q\cdot a^*-b^*\sum_{j=0}^{D-1} P_jY^j = \sum_{j=1}^m g_j^*f_j^*.
\end{equation} 
By Proposition \ref{le:2.2} (ii) this identity holds 
with polynomials $g_j^*\in A_0[\Yv ]$ of degree in $\Yv$ at most
$(2\max (d^* ,D))^{2^t}\leq (2d^*)^{t2^t}$, where possibly we have to multiply
$(P_0\kdots P_{D-1},Q)$ with a non-zero element from $A_0$. 
Now completely similarly as in 
the proof of Lemma \ref{le:3.2}, one can rewrite \eqref{3.6} 
as a system of linear equations over $K_0$
and then apply Corollary \ref{co:2.3}. 
It follows that \eqref{3.5} is satisfied by $P_0\kdots P_{D-1},Q\in A_0$ 
with $Q\not= 0$ and 
\begin{eqnarray*}
&&\deg P_i,\, \deg Q\leq (2d^*)^{\expr},\\
&&h(P_i),\, h(Q)\leq (2d^*)^{\expr}(h^*+1)\ \ (i=0\kdots D-1).
\end{eqnarray*}
By dividing $P_0\kdots P_{D-1},Q$ by their gcd and using Lemma \ref{le:3.0} we obtain
$P_{\alpha,0}\kdots P_{D-1,\alpha},Q_{\alpha}\in A_0$
satisfying both \eqref{3.4} and
\begin{eqnarray*}
&&\deg P_{i,\alpha },\, \deg Q_{\alpha}\leq (2d^*)^{\expr},\\
&&h(P_{i,\alpha }),\, h(Q_{\alpha})\leq (2d^*)^{\expr}(h^*+1)\ \ 
(i=0\kdots D-1).
\end{eqnarray*}
\end{proof}

\begin{lemma}\label{le:3.4a}
Let $\alpha_1\kdots \alpha_n\in K^*$. For $i=1\kdots n$, let
$(a_i,b_i)$ be a pair of representatives for $\alpha_i$, with
$a_i,b_i\in\Zz [X_1\kdots X_r]$, $b_i\not\in I$. Put
\begin{eqnarray*}
d^{**}&:=&\max (d,\deg a_1,\deg b_1\kdots \deg a_n,\deg b_n),
\\[0.1cm]
h^{**}&:=&\max \big(h,h(a_1),h(b_1)\kdots h(a_n),h(b_n)\big).
\end{eqnarray*}
Then there is a non-zero $f\in A_0$ such that
\begin{eqnarray}
\label{3.x1}
&&A\subseteq A_0[y,f^{-1}],\  \alpha_1\kdots\alpha_n\in A_0[y,f^{-1}]^*,
\\[0.15cm]
\label{3.x2}
&&\deg f\leq (n+1)(2d^{**})^{\expr},\  h(f)\leq (n+1)(2d^{**})^{\expr}(h^{**}+1).
\end{eqnarray}
\end{lemma}

\begin{proof} 
Take
\[
f:=\prod_{i=1}^t Q_{y_i}\cdot \prod_{j=1}^n\big( Q_{\alpha_i}Q_{\alpha_i^{-1}}\big).
\]
Since in general, $Q_{\beta}\beta\in A_0[y]$ for $\beta\in K^*$, we have
$f\beta\in A_0[y]$ for $\beta =y_1\kdots y_t,\alpha_1,\alpha_1^{-1}\kdots\alpha_n,\alpha_n^{-1}$.
This implies \eqref{3.x1}.
The inequalities \eqref{3.x2} follow at once from Lemmas \ref{le:3.4} and \ref{le:3.0}.
\end{proof} 

\begin{lemma}\label{le:3.5}
Let $\lambda\in K^*$ and let $\ve$ be a non-zero element of $A$.
Let $(a,b)$ with $a,b\in\Zz [X_1\kdots X_r]$ be a pair of representatives for $\lambda$.
Put
\begin{eqnarray*}
d_0&:=&\max (\deg f_1\kdots \deg f_m ,\deg a,\deg b, \Deg \lambda \ve ),
\\[0.15cm]
h_0&:=&\max \big(h(f_1)\kdots h(f_m) ,h(a),h(b), \Height ( \lambda\ve)\,\big).
\end{eqnarray*}
Then $\ve$ has a representative $\widetilde{\ve}\in\Zz [X_1\kdots X_r]$ such that
\[
\deg\widetilde{\ve}\leq (2d_0)^{\exp O(r\log^* r)}(h_0+1),\ \
h(\widetilde{\ve})\leq (2d_0)^{\exp O(r\log^* r)}(h_0+1)^{r+1}. 
\]
If moreover $\ve\in A^*$, then $\ve^{-1}$ has a representative $\widetilde{\ve}'\in\Zz [X_1\kdots X_r]$
with
\[
\deg \widetilde{\ve}'\leq (2d_0)^{\exp O(r\log^* r)}(h_0+1),\ \
h(\widetilde{\ve}')\leq (2d_0)^{\exp O(r\log^* r)}(h_0+1)^{r+1}. 
\]
\end{lemma}

\begin{proof}
In case that $q>0$, we identify $z_i$ with $X_i$ and view elements of $A_0$ 
as polynomials in $\Zz [X_1\kdots X_q]$.
Put $Y:=\GG_0\cdot (\sum_{i=1}^t a_iX_{q+i})$.
We have
\begin{equation}\label{3.9}
\lambda\ve =Q^{-1}\sum_{i=0}^{D-1} P_iy^i
\end{equation}
with $P_0\kdots P_{D-1},Q\in A_0$
and $\gcd (P_0\kdots P_{D-1},Q)=1$.
According to \eqref{3.9}, $\widetilde{\ve}\in\Zz [X_1\kdots X_r]$ 
is a representative for $\ve$ if and only if
there are $g_1\kdots g_m\in\Zz [X_1\kdots X_r]$ such that
\begin{equation}\label{3.10}
\widetilde{\ve}\cdot (Q\cdot a) +\sum_{i=1}^m g_if_i =
b\sum_{i=0}^{D-1} P_iY^i.
\end{equation}
We may view \eqref{3.10} as an inhomogeneous linear equation in the unknowns
$\widetilde{\ve},g_1\kdots g_m$. 
Notice that by Lemmas \ref{le:3.1}--\ref{le:3.4}
the degrees and logarithmic heights of $Qa$
and $b\sum_{i=0}^{D-1} P_iY^i$
are all bounded above by
$(2d_0)^{\expr}$, $(2d_0)^{\expr}(h_0+1)$, respectively. Now Proposition \ref{le:2.5}
implies that \eqref{3.10} has a solution with upper bounds for $\deg\widetilde{\ve}$,
$h(\widetilde{\ve})$ as stated in the lemma.

Now suppose that $\ve\in A^*$. 
Again by \eqref{3.9}, 
$\widetilde{\ve}'\in\Zz [X_1\kdots X_r]$ is a representative for $\ve^{-1}$
if and only if there
are $g_1'\kdots g_m'\in\Zz [X_1\kdots X_r]$ such that
\[
\widetilde{\ve}'\cdot b\sum_{i=0}^{D-1} P_iY^i +\sum_{i=1}^m g_i'f_i
=Q\cdot a.
\]
Similarly as above, this equation has a solution with upper bounds for 
$\deg\widetilde{\ve}'$, $h(\widetilde{\ve}')$ as stated in the lemma.
\end{proof}

Recall that we have defined $A_0=\Zz [z_1\kdots z_q]$, $K_0=\Qq (z_1\kdots z_q)$ if $q>0$
and $A_0=\Zz$, $K_0=\Qq$ if $q=0$, and that in the case $q=0$, 
degrees and $\Deg$-s are always zero.
Theorem \ref{th:1.1} can be deduced from the following Proposition,
which makes sense also if $q=0$.
The proof of this Proposition is given in Sections \ref{4}--\ref{6}.

\begin{proposition}\label{pr:3.6}
Let $f\in A_0$ with $f\not=0$, and let
\[
\FF =X^D+\FF_1X^{D-1}+\cdots +\FF_D\in A_0[X]\ \ (D\geq 1)
\]
be the minimal polynomial of $y$ over $K_0$.
Let $d_1\geq 1$, $h_1\geq 1$ and suppose 
\[
\max (\deg f ,\deg\FF_1\kdots\deg\FF_D)\leq d_1,\ \ 
\max (h(f) ,h(\FF_1)\kdots h(\FF_D))\leq h_1.
\]
Define the domain $B:=A_0[y,f^{-1}]$.
Then for each pair
$(\ve_1,\eta_1)$ with 
\begin{equation}\label{3.12}
\ve_1+\eta_1=1,\ \ \ \ve_1,\eta_1\in B^*
\end{equation}
we have
\begin{eqnarray}
\label{3.13}
&&\Deg \ve_1 ,\Deg \eta_1 \leq 4qD^2\cdot d_1,
\\[0.15cm]
\label{3.14}
&&\Height (\ve_1) , \Height (\eta_1)\leq \exp O\Big( 2D(q+d_1)\log^* \{2D(q+d_1)\}+Dh_1\Big).
\end{eqnarray}
\end{proposition}

\begin{proof}[Proof of Theorem \ref{th:1.1}]
Let $a,b,c\in A$ be the coefficients of \eqref{1.3},
and $\widetilde{a},\widetilde{b},\widetilde{c}$ the representatives for $a,b,c$
from the statement of Theorem \ref{th:1.1}.
By Lemma \ref{le:3.4a}, there exists non-zero $f\in A_0$ such that
that $A\subseteq B:=A_0[y,f^{-1}]$, $a,b,c\in B^*$, and moreover,
$\deg f\leq (2d)^{\expr}$ and $h(f)\leq (2d)^{\expr}(h+1)$.
By Corollary \ref{co:3.3} we have the same type of upper bounds for the degrees and logarithmic heights
of $\FF_1\kdots \FF_D$. So in Proposition \ref{pr:3.6} we may take
$d_1= (2d)^{\expr}$, $h_1=(2d)^{\expr}(h+1)$. 
Finally, by Lemma \ref{le:3.1} we have $D\leq d^t$.

Let $(\ve ,\eta )$ be a solution of \eqref{1.1} and put $\ve_1:=a\ve /c$, $\eta_1:= b\eta/c$.
By Proposition \ref{pr:3.6} we have
\[
\Deg \ve_1\leq 4qd^{2t} (2d)^{\expr}\leq (2d)^{\expr},\ \ 
\Height (\ve_1)\leq \exp \Big( (2d)^{\expr}(h+1)\Big).
\]
We apply Lemma \ref{le:3.5} with $\lambda =a/c$. Notice that $\lambda$ is represented by
$(\widetilde{a},\widetilde{c})$. By assumption, $\widetilde{a}$ and $\widetilde{c}$ 
have degrees at most $d$
and logarithmic heights at most $h$. Letting $\widetilde{a},\widetilde{c}$ play the role
of $a,b$ in Lemma \ref{le:3.5}, we see that in that lemma we may take
$h_0=\exp \Big( (2d)^{\expr}(h+1)\Big)$ and $d_0=(2d)^{\expr}$.
It follows that $\ve ,\ve^{-1}$ have representatives $\widetilde{\ve}$,
$\widetilde{\ve}'\in\Zz [X_1\kdots X_r]$ such that
\[
\deg \widetilde{\ve},\, \deg \widetilde{\ve}',\, h(\widetilde{\ve}),\, h(\widetilde{\ve}')
\leq  \exp \Big( (2d)^{\expr}(h+1)\Big).
\]
We observe here that the upper bound for $\Height (\ve_1)$ dominates by far the other terms
in our estimation.
In the same manner one can derive similar upper bounds for the degrees and logarithmic heights
of representatives for $\eta$ and $\eta^{-1}$.
This completes the proof of Theorem \ref{th:1.1}.
\end{proof}

Proposition \ref{pr:3.6} is proved in Sections \ref{4}--\ref{6}.
In Section \ref{4} we deduce the degree bound \eqref{3.13}.
Here, our main tool is Mason's effective
result on $S$-unit equations over function fields \cite[1983]{Mas83}.
In Section \ref{5} we work out a more precise version of an effective specialization
argument of Gy\H{o}ry \cite[1983]{Gy83}, \cite[1984]{Gy84}. In Section \ref{6}
we prove \eqref{3.14} by combining the specialization argument from Section \ref{5}
with a recent effective result for $S$-unit equations
over number fields, due to Gy\H{o}ry an Yu \cite[2006]{GyYu06}.

\section{Bounding the degree}\label{4}

We start with recalling some results on function fields in one variable.
Let $\kv$ be an algebraically closed field of characteristic $0$
and let $z$ be transcendental over $\kv$. 
Let $K$ be a finite extension of $\kv (z)$.
Denote by $g_{K/\kv}$ the genus of $K$,
and by $M_K$ the collection
of  valuations of $K/\kv$, i.e, the valuations of $K$ with value group $\Zz$ 
which are trivial on $\kv$. 
Recall that these valuations satisfy the sum formula 
\[
\sum_{v\in M_K} v(x)=0\ \ \mbox{for $x\in K^*$.}
\]
As usual, for a finite subset $S$ of $M_K$ the group of $S$-units of $K$ is given by
\[
O_S^*=\{ x\in K^*:\, v(x)=0\ \mbox{for } v\in M_K\setminus S\}.
\]
The (homogeneous) height of $\xv =(x_1\kdots x_n)\in K^n$ relative to $K/\kv$ is defined by
\[
H_K(\xv )=H_K(x_1\kdots x_n):=-\sum_{v\in M_K}\min (v(x_1)\kdots v(x_n)).
\]
By the sum formula, 
\begin{equation}\label{4.0}
H_K(\alpha \xv )=H_K(\xv )\ \ \mbox{for $\alpha\in K^*$.}
\end{equation}
The height of $x\in K$ relative to $K/\kv$ is defined by
\[ 
H_K(x):= H_K(1,x)=-\sum_{v\in M_K}\min (0,v(x)).
\]
If $L$ is a finite extension of $K$, we have
\begin{equation}\label{4.1}
H_L(x_1\kdots x_n)=[L:K]H_K (x_1\kdots x_n)\ \ \mbox{for } (x_1\kdots x_n) \in K^n.
\end{equation}
By $\deg f$ we denote the total degree of $f\in\kv [z]$.
Then for $f_1\kdots f_n\in\kv [z]$ with $\gcd (f_1\kdots f_n)=1$ we have
\begin{equation}\label{4.2}
H_{\kv (z)}(f_1\kdots f_n)=\max (\deg f_1\kdots \deg f_n).
\end{equation}

\begin{lemma}\label{le:4.1}
Let $y_1\kdots y_m\in K$ and suppose that
\[
X^m+f_1X^{m-1}+\cdots +f_m=(X-y_1)\cdots (X-y_m)
\]
for certain $f_1\kdots f_m\in\kv [z]$. Then
\[
[K:\kv (z)]\max (\deg f_1\kdots\deg f_m)=\sum_{i=1}^m H_K(y_i).
\]
\end{lemma}

\begin{proof}
By Gauss' Lemma we have for $v\in M_K$,
\[
\min (v(f_1)\kdots v(f_m))=\sum_{i=1}^m \min (0,v(y_i)).
\]
Now take the sum over $v\in M_K$ and apply \eqref{4.1}, \eqref{4.2}.
\end{proof}

\begin{lemma}\label{le:4.2}
Let $K$ be the splitting field over $\kv (z)$ of 
$F:=X^m+f_1X^{m-1}+\cdots +f_m$, where $f_1\kdots f_m\in \kv [z]$. Then
\[
g_{K/\kv}\leq (d-1)m\cdot \max_{1\leq i\leq m} \deg f_i ,
\]
where $d:=[K:\kv (z)]$.
\end{lemma}

\begin{proof}
This is Lemma H of Schmidt \cite[1978]{Schmidt78}.
\end{proof}

In what follows, the cardinality of a set $S$ is denoted by $|S|$.

\begin{proposition}\label{le:4.3}
Let $K$ be a finite extension of $\kv (z)$ and $S$ be a finite subset of $M_K$.
Then for every solution of
\begin{equation}\label{4.3}
x+y=1\ \ \mbox{in } x,y\in O_S^*\setminus \kv^*
\end{equation}
we have $\max (H_K(x),H_K(y))\leq |S|+2g_{K/\kv}-2$.
\end{proposition}

\begin{proof}
See Mason \cite[1983]{Mas83}.
\end{proof}

We keep the notation from Proposition \ref{pr:3.6}. We may assume that $q>0$ because the case
$q=0$ is trivial. Let as before $K_0=\Qq (z_1\kdots z_q)$,
$K=K_0(y)$, $A_0=\Zz [z_1\kdots z_q]$, $B=\Zz [z_1\kdots z_q,f^{-1},y]$.

Fix $i\in \{ 1\kdots q\}$. Let
$\kv_i:=\Qq(z_1,\ldots , z_{i-1},z_{i+1},\ldots , z_q)$ and
$\overline{\kv_{i}}$ its algebraic closure.  Thus, the domain
$A_0$ is contained in $\overline{\kv_i}[z_i]$.
Let $y^{(1)}=y,\ldots , y^{(D)}$ denote the conjugates of $y$ over
$K_0$. Let $M_i$ denote the splitting
field of the polynomial
$X^{D}+\mathcal{F}_1X^{D-1}+\cdots +\mathcal{F}_D$ 
over $\overline{\kv_i}(z_i)$, i.e.
\[
M_i:=\overline{\kv_i}(z_i,y^{(1)},\ldots ,y^{(D)}).
\]
The subring
\[
B_i:=\overline{\kv_i}[z_i,f^{-1}, y^{(1)},\ldots ,y^{(D)}]
\]
of $M_i$ contains $B=\Zz [z_1,\ldots ,z_q,f^{-1},y]$
as a subring. Put $\Delta_i:=[M_i:\overline{\kv_i}(z_i)]$.

We apply Lemmas \ref{le:4.1}, \ref{le:4.2} and Proposition \ref{le:4.3} with $z_i,\kv_i,M_i$ instead
of $z,\kv,K$. 
Denote by $g_{M_i}$ the genus of $M_i/\overline{\kv_i}$.
The height $H_{M_i}$ is taken with respect to $M_i/\overline{\kv_i}$.
For $g\in A_0$, we denote by $\deg_{z_i} g$ the degree of $g$ in the variable $z_i$.

\begin{lemma}\label{le:4.4}
Let $\alpha\in K$ and denote by $\alpha^{(1)}\kdots \alpha^{(D)}$ the conjugates of $\alpha$ over $K_0$.
Then
\[
\Deg \alpha \leq qD\cdot d_1+\sum_{i=1}^q \Delta_i^{-1}\sum_{j=1}^D H_{M_i}(\alpha^{(j)} ).
\]
\end{lemma}

\begin{proof}
We have 
\[
\alpha =Q^{-1}\sum_{j=0}^{D-1} P_jy^j
\]
for certain $P_0\kdots P_{D-1},Q\in A_0$ with $\gcd (Q,P_0\kdots P_{D-1})=1$.
Clearly,
\begin{equation}\label{4.3a}
\Deg\alpha\leq \sum_{i=1}^q \mu_i,\ \ \mbox{where } 
\mu_i:=\max (\deg_{z_i} Q,\deg_{z_i} P_0\kdots \deg_{z_i} P_{D-1}).
\end{equation}
Below, we estimate $\mu_1\kdots \mu_q$ from above.
We fix $i\in\{ 1\kdots q\}$ and use the notation introduced above.

Obviously,
\[
\alpha^{(k)} =Q^{-1}\sum_{j=0}^{D-1} P_j\cdot (y^{(k)})^j \ \mbox{for } k=1\kdots D .
\]
Let $\Omega$ be the $D\times D$-matrix with rows
\[
(1\kdots 1),\, (y^{(1)}\kdots y^{(D)})\kdots \big( (y^{(1)})^{D-1}\kdots (y^{(D)})^{D-1}\big).
\]
By Cramer's rule, $P_j/Q=\delta_j/\delta$, where $\delta =\det\Omega$,
and $\delta_j$ is the determinant of the matrix obtained by replacing the $j$-th row of $\Omega$
by $(\alpha^{(1)}\kdots\alpha^{(D)})$.

Gauss' Lemma implies that $\gcd(P_0\kdots P_{D-1},Q)=1$ in the ring
in $\kv_i[z_i]$. By \eqref{4.2} (with $z_i$ in place of $z$) we have
\begin{eqnarray*}
\mu_i &=&
\max (\deg_{z_i}Q,\deg_{z_i} P_0\kdots \deg_{z_i}P_{D-1})
\\[0.15cm]
&=&
H_{\overline{\kv}(z_i)}(Q,P_0\kdots P_{D-1}).
\end{eqnarray*}
Using $[M_i:\overline{\kv_i}(z_i)]=\Delta_i$, the identities \eqref{4.1}, \eqref{4.0} 
(with $z_i$ instead of $z$)
and the fact that 
$(\delta ,\delta_1\kdots \delta_D)$ is a scalar multiple of 
$(Q,P_0\kdots P_{D-1})$ we obtain
\begin{equation}
\label{4.4}
\Delta_i\mu_i =
H_{M_i}(Q,P_0\kdots P_{D-1})
=H_{M_i}(\delta ,\delta_1\kdots \delta_D).
\end{equation}
We bound from above the right-hand side.
A straightforward estimate yields 
that for every valuation $v$ of $M_i/\overline{\kv_i}$,
\begin{eqnarray*}
&&-\min (v(\delta ),v(\delta_1)\kdots v(\delta_D))
\\[0.15cm]
&&\qquad\quad\leq 
-D\sum_{j=1}^D \min (0,v(y^{(j)}))-\sum_{j=1}^D\min (0,v(\alpha^{(j)})).
\end{eqnarray*}
Then summation over $v$ and an application of Lemma \ref{le:4.1} lead to
\begin{eqnarray*}
H_{M_i}(\delta ,\delta_1\kdots \delta_D )&\leq& D\sum_{j=1}^D H_{M_i}(y^{(j)})+\sum_{j=1}^D H_{M_i}(\alpha^{(j)}),
\\[0.1cm]
&\leq& 
D\Delta_i\max (\deg_{z_i} \FF_1\kdots\deg\FF_D)+\sum_{j=1}^D H_{M_i}(\alpha^{(j)})
\\[0.1cm]
&\leq& \Delta_i\cdot Dd_1 + \sum_{j=1}^D H_{M_i}(\alpha^{(j)}),
\end{eqnarray*}
and then a combination with \eqref{4.4} gives
\[
\mu_i\leq Dd_1+\Delta_i^{-1}\sum_{j=1}^D H_{M_i}(\alpha^{(j)}).
\]
Now these bounds for $i=1\kdots q$ together with \eqref{4.3a} imply our Lemma.
\end{proof}

\begin{proof}[Proof of \eqref{3.13}]
We fix again $i\in\{ 1\kdots q\}$ and use the notation introduced above.
By Lemma \ref{le:4.2}, applied with $\kv_i,z_i,M_i$ instead of $\kv, z,K$ and with
$F=\mathcal{F}=X^{D}+\mathcal{F}_1X^{D-1}+\cdots +\mathcal{F}_D$, we have
\begin{equation}\label{4.6}
g_{M_i}\leq(\Delta_i-1)D\max_j\deg_{z_i}(\mathcal{F}_j)\leq (\Delta_i-1)\cdot Dd_1.
\end{equation} 
Let $S$ denote the subset of valuations $v$ of
$M_i/\overline{\kv_i}$ such that $v(z_i)<0$ or $v(f)>0$.
Each valuation of $\overline{\kv_i}(z_i)$ can be extended to at most
$[M_i:\overline{\kv_i}(z_i)]=\Delta_i$ valuations of $M_i$.
Hence $M_i$ has at most $\Delta_i$ valuations $v$ with $v(z_i)<0$ and at most
$\Delta_i\deg f$ valuations with $v(f)>0$.
Thus,
\begin{equation}\label{4.7}
|S|\leq \Delta_i+\Delta_i\deg_{z_i} f\leq\Delta_i(1+\deg f)\leq \Delta_i(1+d_1).
\end{equation}

Every $\alpha\in M_i$ which is integral over 
$\overline{\kv}_i[z_i,f^{-1}]$ belongs to $O_S$.
The elements  
$y^{(1)}\kdots y^{(D)}$ belong to $M_i$ and are integral over $A_0=\Zz [z_1\kdots z_q]$
so they certainly belong to $O_S$. As a consequence, 
the elements of $B$
and their conjugates over $\Qq (z_1\kdots z_q)$ belong to $O_S$.
In particular, if $\ve_1,\eta_1\in B^*$ and $\ve_1+\eta_1=1$, then
\begin{equation}\label{4.5}
\ve_1^{(j)}+\eta_1^{(j)}=1,\  \ve_1^{(j)},\eta_1^{(j)}\in O_S^*\ \ \mbox{for } j=1\kdots D .
\end{equation}

We apply Proposition \ref{le:4.3} and insert the upper bounds \eqref{4.6}, \eqref{4.7}.
It follows that for $j=1\kdots D$ we have either $\ve_1^{(j)}\in\overline{\kv_i}$ or 
\[
H_{M_i}(\ve_1^{(j)})\leq |S|+2g_{M_i}-2\leq 3\Delta_i\cdot Dd_1;
\]
in fact the last upper bound is valid also if $\ve_1^{(j)}\in\overline{\kv_i}$.
Together with Lemma \ref{le:4.4} this gives
\[
\Deg \ve_1\leq qDd_1 + qD\cdot 3Dd_1\leq 4qD^2d_1.
\]
For $\Deg \eta_1$ we derive the same estimate. This proves \eqref{3.13}.
\end{proof}

\section{Specializations}\label{5}

In this section we prove some results about specialization homomorphisms from 
the domain $B$ from Proposition \ref{pr:3.6} to $\OQq$.
We start with some notation and some preparatory lemmas.

The set of places of $\Qq$ is $M_{\Qq}=\{ \infty\}\cup\{ {\rm primes}\}$. By $|\cdot |_{\infty}$
we denote the ordinary absolute value on $\Qq$ and by $|\cdot |_p$ ($p$ prime) the $p$-adic absolute
value, with $|p|_p=p^{-1}$.
More generally, let $L$ be an algebraic number field and denote by $M_L$ its set of places.
Given $v\in M_L$, we define the absolute value $|\cdot |_v$ in such a way that its restriction to $\Qq$
is $|\cdot |_p$ if $v$ lies above $p\in M_{\Qq}$. These absolute values satisfy
the product formula $\prod_{v\in M_L} |x|_v^{d_v}=1$ for $x\in L^*$, where
$d_v:=[L_v:\Qq_p]/[L:\Qq ]$.

The (absolute logarithmic) height of $\xv =(x_1\kdots x_m)\in L^m\setminus\{ \nullv\}$ is defined by
\[
h(\xv )=h(x_1\kdots x_m)=\log\prod_{v\in M_L}\big(\max (|x_1|_v\kdots |x_m|_v)\big)^{d_v}.
\]
By the product formula, $h(\alpha\xv )=h(\xv )$ for $\alpha\in L^*$. Moreover,
$h(\xv )$ depends only on $\xv$ and not on the choice of the field $L$ such that $\xv\in L^m$.
So it defines a height on $\OQq^m\setminus\{\nullv\}$. 
The (absolute logarithmic) height of $\alpha\in\OQq$ is defined
by $h(\alpha ):=h((1,\alpha ))$. In case that $\alpha\in L$ we have
\[
h(\alpha )=\log\prod_{v\in M_L} \max (1,|\alpha |_v^{d_v}).
\]
For $\av =(a_1\kdots a_m)\in\Zz^m$ with $\gcd (a_1\kdots a_m)=1$ we have
\begin{equation}\label{5.6a}
h(\av )=\log\max (|a_1|\kdots |a_m|).
\end{equation}
It is easy to verify that for $a_1\kdots a_m,b_1\kdots b_m\in\OQq$,
\begin{equation}\label{5.7a}
h(a_1b_1+\cdots +a_mb_m)\leq h(1,a_1\kdots a_m)+ h(1,b_1\kdots b_m)+\log m.
\end{equation}

Let $G$ be a polynomial with coefficients in $L$. If $a_1\kdots a_r$ are the
non-zero coefficients of $G$, we put $|G|_v:=\max (|a_1|_v\kdots |a_r|_v)$
for $v\in M_L$. For a polynomial $G$ with coefficients in $\Zz$ we define
$h(G):=\log |G|_{\infty}$.

We start with four auxiliary results that are used in the construction of our specializations.

\begin{lemma}\label{le:5.1}
Let $m\geq 1$, $\alpha_1\kdots \alpha_m\in\OQq$ and suppose that 
$G(X):=\prod_{i=1}^m (X-\alpha_i)\in\Zz [X]$.
Then
\[
|h(G)-\sum_{i=1}^m h(\alpha _i)|\leq m.
\]
\end{lemma}

\begin{proof}
See Bombieri and Gubler \cite[Theorem 1.6.13, pp. 28]{BoGu06}.
\end{proof}

\begin{lemma}\label{le:5.2}
Let $m\geq 1$, let $\alpha_1\kdots\alpha_m\in\OQq$ be distinct and suppose that
$G(X):=\prod_{i=1}^m (X-\alpha_i)\in\Zz [X]$. Let $q,p_0\kdots p_{m-1}$ be integers 
with 
\[
\gcd (q,p_0\kdots p_{m-1})=1,
\]
and put
\[
\beta_i:=\sum_{j=0}^{m-1} (p_j/q)\alpha_i^j\ \ (i=1\kdots m).
\]
Then
\[
\log \max (|q|,|p_0|\kdots |p_{m-1}|)\leq 2m^2 + (m-1)h(G)+\sum_{j=1}^m h(\beta_j).
\]
\end{lemma}

\begin{proof}
For $m=1$ the assertion is obvious, so we assume $m\geq 2$.
Let $\Omega$ be the $m\times m$ matrix with rows $(\alpha_1^i\kdots \alpha_m^i)$ ($i=0\kdots m-1)$.
By Cramer's rule we have $p_i/q =\delta_i/\delta$ ($i=0\kdots m-1$), 
where $\delta =\det\Omega$ and $\delta_i$
is the determinant of the matrix, obtained by replacing the $i$-th row of $\Omega$ by $(\beta_1\kdots\beta_m)$.
Put $\mu :=\log \max (|q|,|p_0|\kdots |p_{m-1}|)$. Then by \eqref{5.6a}, 
\[
\mu = h(q,p_0\kdots p_{m-1})=h(\delta ,\delta_0\kdots\delta_{m-1}).
\]
Let $L=\Qq (\alpha_1\kdots\alpha_m)$. By Hadamard's inequality for the infinite places and the ultrametric
inequality for the finite places, we get
\[
\max (|\delta |_v,|\delta_1|_v\kdots|\delta_m|_v)\leq 
c_v\prod_{i=1}^m\max (1,|\alpha_i|_v)^{m-1}\max (1,|\beta_i|_v)
\]
for $v\in M_L$, where $c_v=m^{m/2}$ if $v$ is infinite and $c_v=1$ if $v$ is finite. By taking the
product over $v\in M_L$ and then logarithms, it follows that
\[
\mu\leq \half m\log m + \sum_{i=1}^m \big( (m-1)h(\alpha_i)+h(\beta_i)\big).
\]
A combination with Lemma \ref{le:5.1} implies our lemma.
\end{proof}

\begin{lemma}\label{le:5.3}
Let $g\in\Zz [z_1\kdots z_q]$ be a non-zero polynomial of degree $d$ and $\NN$
a subset of $\Zz$ of cardinality $>d$. Then
\[
|\{ \uv\in\NN^q :\, g(\uv )=0\}|\leq d|\NN |^{q-1}.
\]
\end{lemma}

\begin{proof} We proceed by induction on $q$. For $q=1$ the assertion is clear.
Let $q\geq 2$. Write $g=\sum_{i=0}^{d_0} g_i(z_1\kdots z_{q-1})z_q^i$ with
$g_i\in \Zz [z_1\kdots z_{q-1}]$ and $g_{d_0}\not= 0$. 
Then $\deg g_{d_0}\leq d-d_0$.
By the induction hypothesis, there are at most $(d-d_0)|\NN |^{q-2}\cdot |\NN |$ tuples
$(u_1\kdots u_q)\in\NN^q$ with $g_{d_0}(u_1\kdots u_{q-1})=0$. Further, 
there are at most $|\NN |^{q-1}\cdot d_0$ tuples $\uv\in\NN^q$ with
$g_{d_0}(u_1\kdots u_{q-1})\not= 0$ and $g(u_1\kdots u_q)=0$. 
Summing these two quantities implies that $g$ has at most $d|\NN |^{q-1}$ zeros in $\NN^q$.
\end{proof}

\begin{lemma}\label{le:5.4}
Let $g_1,g_2\in\Zz [z_1\kdots z_q]$ be two non-zero polynomials of degrees $D_1,D_2$, respectively,
and let $N$ be an integer $\geq\max(D_1,D_2)$. Define
\[
\SS :=\{ \uv\in\Zz^q:\ |\uv |\leq N,\, g_2(\uv )\not =0\}.
\]
Then $\SS$ is non-empty, and
\begin{eqnarray}\label{5.4}
&&|g_1|_p\leq (4N)^{qD_1(D_1+1)/2}\max\{ |g_1(\uv )|_p:\ \uv\in\SS\}
\\[0.15cm]
\nonumber
&&\hspace*{4cm}
\mbox{for } 
p\in M_{\Qq}=\{\infty\}\cup\{ {\rm primes}\}.
\end{eqnarray}
\end{lemma}

\begin{proof}
Put $C_p:=\max\{ |g_1(\uv )|_p:\ \uv\in\SS\}$ for $p\in M_{\Qq}$.
We proceed by induction on $q$, starting with $q=0$. In the case $q=0$ we interpret $g_1,g_2$ as 
non-zero constants with $|g_1|_p=C_p$ for $p\in M_{\Qq}$. Then the lemma is trivial. Let $q\geq 1$. Write
\[
g_1=\sum_{j=0}^{D_1'} g_{1j}(z_1\kdots z_{q-1})z_q^j,\ \ g_2=\sum_{j=0}^{D_2'} g_{2j}(z_1\kdots z_{q-1})z_q^j,
\]
where $g_{1,D_1'},g_{2,D_2'}\not= 0$. By the induction hypothesis, the set
\[
\SS ':=\{ \uv ' \in\Zz^{q-1}:\, |\uv '|\leq N,\ g_{2,D_2'}(\uv ' )\not= 0\}
\]
is non-empty and moreover,
\begin{equation}\label{5.5}
\max_{0\leq j\leq D_1'} |g_{1j}|_p\leq (4N)^{(q-1)D_1(D_1+1)/2}C_p'\ \ \mbox{for } p\in M_{\Qq}
\end{equation}
where
\[
C_p':=\max\{ |g_{1j}(\uv ')|_p:\, \uv '\in\SS',\, j=0\kdots D_1'\}.
\]

We estimate $C_p'$ from above in terms of $C_p$. Fix $\uv '\in\SS '$.
There are at least  
$2N+1-D_2'\geq D_1'+1$ integers $u_q$ with $|u_q|\leq N$ such that
$g_2(\uv ',u_q)\not= 0$. Let $a_0\kdots a_{D_1'}$ be distinct integers from this set.
By Lagrange's interpolation formula,
\begin{eqnarray*}
g_1(\uv ' ,X)&=&\sum_{j=0}^{D_1'} g_{1j}(\uv ')X^j
\\[0.15cm]
&=&
\sum_{j=0}^{D_1'} g_1(\uv ',a_j)\prod_{\stackrel{i=0}{i\not= j}}^{D_1'}\frac{X-a_i}{a_j-a_i}.
\end{eqnarray*}
From this we deduce
\begin{eqnarray*}
\max_{0\leq j\leq D_1'} |g_{1j}(\uv ')|&\leq& 
C_{\infty}\sum_{j=0}^{D_1'}\prod_{\stackrel{i=0}{i\not= j}}^{D_1'}\frac{1+|a_i|}{|a_j-a_i|} 
\\[0.15cm]
&\leq& C_{\infty}(D_1'+1)(N+1)^{D_1'}\leq (4N)^{D_1'(D_1'+1)/2}C_{\infty}.
\end{eqnarray*}
Now let $p$ be a prime and put $\Delta :=\prod_{1\leq i<j\leq D_1'}|a_j-a_i|$.
Then
\[
\max_{0\leq j\leq D_1'} |g_{1j}(\uv ')|_p\leq C_p|\Delta |_p^{-1}\leq \Delta C_p\leq (4N)^{D_1'(D_1'+1)/2}C_p.
\]
It follows that $C_p'\leq (4N)^{D_1'(D_1'+1)/2}C_p$ for $p\in M_{\Qq}$.
 A combination with \eqref{5.5} gives \eqref{5.4}.
\end{proof}

We now introduce our specializations $B\to\OQq$ and prove some properties.
We assume $q>0$ and apart from that
keep the notation and assumptions from Proposition \ref{pr:3.6}.
In particular, $A_0=\Zz [z_1\kdots z_q ]$, $K_0=\Qq (z_1\kdots z_q)$ and 
\[
K=\Qq (z_1\kdots z_q,y),\ \ B=\Zz [z_1\kdots z_q,f^{-1},y],
\]
where $f$ is a non-zero element of $A_0$, $y$ is integral over $A_0$, and $y$ has minimal polynomial
\[
\FF :=X^D+\FF_1X^{D-1}+\cdots +\FF_D\in A_0[X]
\]
over $K_0$. In the case $D=1$, we take $y=1$, $\FF =X-1$.

To allow for other applications (e.g., Lemma \ref{le:7.2} below), 
we consider a more general situation 
than what is needed for the proof of Proposition \ref{pr:3.6}.
Let $d_1\geq d_0\geq 1$, $h_1\geq h_0\geq 1$ and assume that
\begin{equation}\label{5.x}
\left\{\begin{array}{l}
\max (\deg \FF_1\kdots\deg \FF_D)\leq d_0,\ \ \max (d_0,\deg f)\leq d_1,
\\[0.15cm]
\max \big(h(\FF_1)\kdots h(\FF_D)\,\big)\leq h_0,\ \ \max (h_0,h(f))\leq h_1.
\end{array}\right.
\end{equation}

Let $\uv =(u_1\kdots u_q)\in\Zz^q$. Then the substitution $z_1\mapsto u_1\kdots z_q\mapsto u_q$
defines a ring homomorphism (specialization)
\[
\irS_{\uv}:\ \alpha\mapsto \alpha (\uv ):\ 
\{ g_1/g_2 :\, g_1,g_2\in A_0,\, g_2(\uv )\not= 0\}\to\Qq .
\] 
We want to extend this to a ring homomorphism from $B$ to $\OQq$ and for this,
we have to impose some restrictions on $\uv$.
Denote by $\Delta_{\FF}$ the discriminant of $\FF$ (with $\Delta_{\FF}:=1$ if $D=\Deg\FF =1$), 
and let
\begin{equation}\label{5.1}
\HH:= \Delta_{\FF}\FF_D\cdot f .
\end{equation}
Then $\HH\in A_0$.
Using that $\Delta_{\FF}$ is a polynomial of degree $2D-2$ with integer coefficients in
$\FF_1\kdots \FF_D$, it follows easily that
\begin{equation}\label{5.2}
\deg \HH \leq (2D-1)d_0+d_1\leq 2Dd_1.
\end{equation}
Now assume that
\begin{equation}\label{5.3}
\HH (\uv )\not= 0.
\end{equation}
Then $f(\uv )\not= 0$ and moreover, the polynomial
\[
\FF_{\uv}:= X^D+\FF_1(\uv )X^{D-1}+\cdots +\FF_D(\uv )
\]
has $D$ distinct zeros which are all different from $0$, say $y_1(\uv )\kdots y_D(\uv )$. 
Thus, for $j=1\kdots D$ the assignment
\[
z_1\mapsto u_1\kdots z_q\mapsto u_q,\ \ y\mapsto y_j(\uv )
\]
defines a ring homomorphism $\irS_{\uv ,j}$ from $B$ to $\OQq$;
in the case $D=1$ it is just $\irS_{\uv}$. 
The image of $\alpha\in B$ under $\irS_{\uv,j}$
is denoted by $\alpha_j(\uv )$. Recall that we may express elements $\alpha$ of $B$ as
\begin{eqnarray}\label{5.6}
&&\alpha =\sum_{i=0}^{D-1} (P_i/Q)y^i
\\[0.1cm]
\nonumber
&&\qquad\mbox{with $P_0\kdots P_{D-1},Q\in A_0$, $\gcd (P_0\kdots P_{D-1},Q)=1$.}
\end{eqnarray}
Since $\alpha\in B$, the denominator $Q$ must divide a power of $f$, hence $Q(\uv )\not= 0$. So we have
\begin{equation}\label{5.7}
\alpha_j(\uv )=\sum_{i=0}^{D-1} (P_i(\uv )/Q(\uv ))y_j(\uv )^i\ \ (j=1\kdots D).
\end{equation}
It is obvious that $\irS_{\uv ,j}$ is the identity on $B\cap\Qq$. 
Thus, if $\alpha\in B\cap\OQq$, then
$\irS_{\uv ,j}(\alpha )$ has the same minimal polynomial as $\alpha$  and so it is conjugate to
$\alpha$.

For $\uv=(u_1\kdots u_q)\in\Zz^q$, we put $|\uv |:=\max (|u_1|\kdots |u_q|)$.
It is easy to verify that for any $g\in A_0$, $\uv\in\Zz^q$,
\begin{equation}\label{5.8}
\log |g(\uv )|\leq q\log \deg g +h(g)+\deg g\log \max(1,|\uv |).
\end{equation}
In particular, 
\begin{equation}\label{5.8b}
h(\FF_{\uv})\leq q\log d_0+h_0+d_0\log\max (1,|\uv |)
\end{equation}
and so by Lemma \ref{le:5.2} (ii),
\begin{equation}\label{5.8a}
\sum_{j=1}^D h(y_j(\uv ))\leq D+ q\log d_0+h_0+d_0\log\max (1,|\uv |).
\end{equation}

Define the algebraic number fields $K_{\uv ,j}:=\Qq (\yv_j(\uv ))$ $(j=1\kdots D)$.
Denote by $\Delta_L$ the discriminant of an algebraic number field $L$.
We derive an upper bound for the discriminant $\Delta_{K_{\uv,j}}$ of $K_{\uv,j}$.

\begin{lemma}\label{le:5.7}
Let $\uv\in\Zz^q$ with $\HH (\uv )\not= 0$. Then for $j=1\kdots D$
we have $[K_{\uv ,j}:\Qq ]\leq D$ and
\[
|\Delta_{K_{\uv ,j}}|\leq D^{2D-1}\left( d_0^q\cdot e^{h_0}\max (1,|\uv |)^{d_0}\right)^{2D-2}.
\]
\end{lemma}

\begin{proof}
Let $j\in\{ 1\kdots D\}$. The estimate for the degree is obvious.
To estimate the discriminant, 
let $\PP_j$ be the monic minimal polynomial of $y_j(\uv )$.
Then $\Delta_{K_{\uv ,j}}$ divides the discriminant $\Delta_{\PP_j}$ of $\PP_j$.
Using the expression of the discriminant of a monic polynomial as the product of the squares
of the differences of its zeros, one easily shows that $\Delta_{\PP_j}$ divides
$\Delta_{\FF_{\uv}}$ in the ring of algebraic integers and so also in $\Zz$. Therefore,
$\Delta_{K_{\uv},j}$ divides $\Delta_{\FF_{\uv}}$ in $\Zz$.

It remains to estimate from above the discriminant of $\FF_{\uv}$. By, e.g.,
Lewis and Mahler \cite[bottom of p. 335]{LeMa61}, we have 
\[
|\Delta_{\FF_{\uv}}|\leq D^{2D-1}|\FF_{\uv}|^{2D-2},
\]
where $|\FF_{\uv}|$ denotes the maximum of the absolute values of the coefficients
of $\FF_{\uv}$. By \eqref{5.8b}, this is bounded above by
$d_0^qe^{h_0}\max(1,|\uv |)^{d_0}$, so
\[
|\Delta_{\FF_{\uv}}|\leq D^{2D-1}\big(d_0^qe^{h_0}\max (1,|\uv |)^{d_0}\big)^{2D-2}.
\]
This implies our lemma.
\end{proof}

We finish with two lemmas, which relate the height of $\alpha\in B$
to the heights of $\alpha_j(\uv )$ for $\uv\in\Zz^q$.

\begin{lemma}\label{le:5.5}
Let $\uv\in\Zz^q$ with  $\HH (\uv )\not= 0$. Let $\alpha\in B$. Then for $j=1\kdots D$,
\begin{eqnarray*}
&&h(\alpha_j(\uv ))\leq D^2 +q(D\log d_0+\log \Deg\alpha ) + Dh_0 +\Height (\alpha)\, +
\\[0.1cm]
&&\hspace*{6cm} + (Dd_0+\Deg\alpha )\log\max (1,|\uv |).
\end{eqnarray*}
\end{lemma}

\begin{proof}
Let $P_0\kdots P_{D-1},Q$ as in \eqref{5.6}
and write $\alpha_j(\uv )$ as in \eqref{5.7}. By \eqref{5.7a},
\begin{eqnarray}\label{5.9}
&&h(\alpha_j(\uv )) \leq \log D +
\\[0.1cm]
\nonumber
&&\qquad +h\big(1,P_0(\uv )/Q(\uv )\kdots P_{D-1}(\uv )/Q(\uv )\big)+(D-1)h(y_j(\uv )).
\end{eqnarray}
From \eqref{5.8} we infer
\begin{eqnarray*}
&&h(1,P_0(\uv )/Q(\uv )\kdots P_{D-1}(\uv )/Q(\uv ))
\\[0.15cm]
&&\qquad \leq
\log\max (|Q(\uv )|,|P_0(\uv )|\kdots |P_{D-1}(\uv )|)
\\[0.15cm]
&&\qquad\leq q\log \Deg\alpha +\Height (\alpha )+\Deg \alpha\cdot \log \max (1,|\uv |).
\end{eqnarray*}
By combining \eqref{5.9} with this inequality and with \eqref{5.8a},
our lemma easily follows.
\end{proof}

\begin{lemma}\label{le:5.6}
Let $\alpha\in B$, $\alpha\not= 0$, and let $N$ be an integer with 
\[
N\geq\max \big(\Deg\alpha ,\, 2Dd_0+2(q+1)(d_1+1)\, \big).
\]
Then the set
\[
\SS :=\{ \uv\in\Zz^q:\ |\uv |\leq N,\ \HH (\uv )\not= 0\}
\]
is non-empty, and
\[
\Height (\alpha )\leq 5N^4(h_1+1)^2+2D(h_1+1)H
\]
where $H:=\max\{ h(\alpha_j(\uv )) :\, \uv\in\SS ,\, j=1\kdots D\}$. 
\end{lemma}

\begin{proof}
It follows from our assumption on $N$, \eqref{5.2}, and Lemma \ref{le:5.4} 
that $\SS$ is non-empty. We proceed with estimating $\Height (\alpha )$.
 
Let $P_0\kdots P_{D-1},Q\in A_0$ be as in \eqref{5.6}. We analyse $Q$ more closely.
Let
\[
f=\pm p_1^{k_1}\cdots p_m^{k_m}g_1^{l_1}\cdots g_n^{l_n}
\]
be the unique factorization of $f$ in $A_0$, where $p_1\kdots p_m$ are distinct prime numbers,
and $\pm g_1\kdots \pm g_n$ distinct irreducible elements of $A_0$ of positive degree.
Notice that
\begin{eqnarray}\label{5.10}
&m\leq h(f)/\log 2\leq h_1/\log 2,&
\\[0.15cm]
\label{5.11}
&\displaystyle{\sum_{i=1}^n l_ih(g_i)\leq qd_1+h_1,}&
\end{eqnarray}
where the last inequality is a consequence of Lemma \ref{le:5.1}.
Since $\alpha\in B$, the polynomial $Q$ is also composed of $p_1\kdots p_m$, $g_1\kdots g_n$. 
Hence
\begin{equation}\label{5.11a}
Q=a\widetilde{Q}\ \ \mbox{with } a=\pm p_1^{k_1'}\cdots p_m^{k_m'},\  \widetilde{Q}=g_1^{l_1'}\cdots g_n^{l_n'}
\end{equation}
for certain non-negative integers $k_1'\kdots l_n'$. Clearly,
\[
l_1'+\cdots +l_n'\leq\deg Q\leq\Deg \alpha\leq N,
\]
and by Lemma \ref{le:3.0} and \eqref{5.11},
\begin{equation}\label{5.12}
h(\widetilde{Q})\leq q\deg Q +\sum_{i=1}^n l_i'h(g_i)\leq N(q+qd_1+h_1)\leq N^2(h_1+1 ).
\end{equation}
In view of \eqref{5.8}, we have for $\uv\in\SS$, 
\begin{eqnarray*}
\log |\widetilde{Q}(\uv )|&\leq& q\log d_1+h(\widetilde{Q})+\deg Q\log N
\\[0.15cm]
&\leq& \textfrac{3}{2}N\log N +N^2(h_1+1)\leq N^2(h_1+2).
\end{eqnarray*}
Hence
\[
h(\widetilde{Q}(\uv )\alpha_j(\uv ))\leq N^2(h_1+2)+H
\]
for $\uv\in\SS$, $j=1\kdots D$. Further, by \eqref{5.7}, \eqref{5.11} we have
\[
\widetilde{Q}(\uv )\alpha_j(\uv )=\sum_{i=0}^{D-1} (P_i(\uv )/a)y_j(\uv )^i .
\]
Put
\[
\delta (\uv ):=\gcd (a,P_0(\uv )\kdots P_{D-1}(\uv )).
\]
Then by applying Lemma \ref{le:5.2} and then \eqref{5.8b} we obtain
\begin{eqnarray}\label{5.13}
&&\log\left(\frac{\max (|a|, |P_0(\uv )|\kdots |P_{D-1}(\uv )|}{\delta (\uv )}\right)
\\[0.15cm]
\nonumber
&&\leq 2D^2+(D-1)h(\FF_{\uv})+D\big( N^2(h_1+2)+H\big)
\\[0.15cm]
\nonumber
&&\leq 2D^2+(D-1)(q\log d_1 +h_1+d_1\log N)+D\big( N^2(h_1+2)+H\big)
\\[0.15cm]
\nonumber
&& \leq N^3(h_1+2)+DH.
\end{eqnarray}

Our assumption that $\gcd (Q,P_0\kdots P_{D-1})=1$ implies that the gcd of $a$
and the coefficients of $P_0\kdots P_{D-1}$ is $1$. Let $p\in\{ p_1\kdots p_m\}$ be one
of the prime factors of $a$. There is $j\in\{ 0\kdots D-1\}$ such that $|P_j|_p=1$.
Our assumption on $N$ and \eqref{5.2} imply
that $N\geq \max (\deg\HH ,\deg P_j)$. This means that Lemma \ref{le:5.4}
is applicable with $g_1= P_j$ and $g_2=\HH$. It follows that
\[
\max\{ |P_j(\uv )|_p:\ \uv\in\SS\}\geq (4N)^{-qN(N+1)/2}.
\]
That is, there is $\uv_0\in\SS$ with $|P_j(\uv_0 )|_p\geq (4N)^{-qN(N+1)/2}$. 
Hence 
\[
|\delta (\uv_0 )|_p\geq (4N)^{-qN(N+1)/2}.
\]
Together with \eqref{5.13},
this implies
\begin{eqnarray*}
\log |a|_p^{-1}&\leq& \log |a/\delta (\uv_0 )|+\log |\delta (\uv_0 )|_p^{-1}
\\[0.15cm]
&\leq& N^3(h_1+2)+DH +\half N^3\log 4N\leq N^4(h_1+3)+DH.
\end{eqnarray*}
Combining this with the upper bound \eqref{5.10} for the number of prime factors of $a$,
we obtain
\begin{equation}\label{5.14}
\log |a|\leq 2N^4h_1(h_1+3)+2Dh_1\cdot H.
\end{equation}
Together with \eqref{5.11a}, \eqref{5.12}, this implies
\begin{eqnarray}\label{5.15}
h(Q)&\leq& 2N^4h_1(h_1+3)+2Dh_1\cdot H +N^2(h_1+1)
\\[0.15cm]
\nonumber
&\leq& 3N^4(h_1+1)^2+2Dh_1\cdot H.
\end{eqnarray}
Further, the right-hand side of \eqref{5.14} is also an upper bound for $\log \delta (\uv )$,
for $\uv\in\SS$. Combining this with \eqref{5.13} gives 
\begin{eqnarray*}
&&\log\max\{ |P_j(\uv )|:\ \uv\in\SS ,\, j=0\kdots D-1\}
\\[0.15cm]
&&\quad
\leq N^3(h_1+2)+DH + 3N^4(h_1+1)^2+2Dh_1\cdot H
\\[0.15cm]
&&\quad\leq 4N^4(h_1+1)^2+2D(h_1+1)\cdot H.
\end{eqnarray*}
Another application of Lemma \ref{le:5.4} yields
\begin{eqnarray*}
h(P_j)&\leq& \half qN(N+1)\log 4N + 4N^4(h_1+1)^2+2D(h_1+1)\cdot H
\\[0.15cm]
&\leq& 5N^4(h_1+1)^2+2D(h_1+1)\cdot H 
\end{eqnarray*}
for $j=0\kdots D-1$. Together with \eqref{5.15} this gives the upper bound
for $\Height (\alpha )$ from our lemma.
\end{proof}

\section{Completion of the proof of Proposition \ref{pr:3.6}}\label{6}

It remains only to prove the height bound in \eqref{3.14}.
We use an effective result of Gy\H{o}ry and Yu \cite[2006]{GyYu06} on $S$-unit equations
in number fields. To state this, we need some notation.

Let $L$ be an algebraic number field of degree $d_L$.
We denote by $O_L$, $M_L$, $\Delta_L$, $h_L$, $R_L$ 
the ring of integers, set of places, discriminant, class number and regulator of $L$.
The norm of an ideal $\fa$ of $O_L$, i.e., $|O_L/\fa |$, is denoted by $N\fa$.

Further, let $S$ be a finite set of places of $L$,
containing all infinite places. Suppose $S$ has cardinality $s$. 
Recall that the ring of $S$-integers $O_S$ and 
the group of $S$-units $O_S^*$ are given by
\begin{eqnarray*}
O_S&=&\{ x\in L :\, |x|_v\leq 1\ \mbox{for } v\in M_L\setminus S\},\\[0.15cm]
O_S^*&=&\{ x\in L :\, |x|_v=1\ \mbox{for } v\in M_L\setminus S\}.
\end{eqnarray*}
If case that $S$ consists only of the infinite places of $L$, we put $P:=2$, $Q:=2$.
If $S$ contains also finite places, let $\fp_1\kdots\fp_t$ denote the prime ideals
corresponding to the finite places of $S$, and put
\[
P:=\max\{ N\fp_1\kdots N\fp_t \},\ \ \ Q:= N(\fp_1\cdots\fp_t ). 
\]
Further, let $R_S$ denote the $S$-regulator associated with $S$.
In case that $S$ consists only of the infinite places of $L$ it is equal to $R_L$,
while otherwise
\[
R_S =h_SR_L\prod_{i=1}^t\log N\fp_i ,
\]
where $h_S$ is a divisor of $h_L$ whose definition is not important here.
By, e.g., formula (59) of \cite{GyYu06} (which is an easy consequence of formula (2)
of Louboutin \cite[2000]{Lou00}) we have
\[
h_LR_L\leq |\Delta_L|^{1/2}(\log^*|\Delta_L|)^{d_L-1}.
\]
By the inequality of the geometric and arithmetic mean, we have for $t>0$,
\[
\prod_{i=1}^t\log N\fp_i\leq \Big( t^{-1}\log (N\fp_1\cdots N\fp_t))^t\leq (\log Q)^s
\]
and hence,
\begin{equation}\label{6.1}
R_S\leq |\Delta_L|^{1/2}(\log^*|\Delta_L|)^{d_L-1}\cdot (\log^* Q)^s.
\end{equation}
This is clearly true also if $t=0$.

\begin{proposition}\label{le:6.1}
Let $\ve ,\eta$ such that
\begin{equation}\label{6.2}
\ve +\eta =1,\ \ \ve ,\eta\in O_S^*.
\end{equation}
Then
\begin{equation}\label{6.3}
\max (h(\ve ),h(\eta ))\leq c_1PR_S\left( 1+\log ^*R_S/\log P\right),
\end{equation}
where
\[
c_1=\max (1,\pi/d_L)s^{2s+3.5}2^{7s+27}(\log 2s)d_L^{2(s+1)}(\log^* 2d_L)^3.
\]
\end{proposition}

\begin{proof}
This is Theorem 1 of Gy\H{o}ry, Yu \cite{GyYu06} with $\alpha =\beta =1$.
\end{proof}

\begin{proof}[Proof of \eqref{3.14}]
As before, we use $O(\cdot )$ to denote a quantity which is $c\times$ the expression between
the parentheses, where $c$ is an effectively computable absolute constant which may be
different at each occurrence of the $O$-symbol. 

We first consider the case $q>0$.
Let $\ve_1 ,\eta_1$ be a solution of \eqref{3.12}. 
Pick 
$\uv\in\Zz^q$ with $\HH (\uv )\not= 0$,
pick $j\in\{ 1\kdots D\}$ and put $L:=K_{\uv ,j}$.
Further, let the set of places $S$ consist of all infinite places of $L$,
and all finite places of $L$ lying above the rational prime divisors of $f(\uv )$.
Note that $y_j(\uv )$ is an algebraic integer, and $f(\uv )\in O_S^*$.
Hence $\irS_{\uv ,j}(B)\subseteq O_S$ and $\irS_{\uv ,j}(B^*)\subseteq O_S^*$. So
\begin{equation}\label{6.y}
\ve_{1,j}(\uv )+\eta_{1,j}(\uv )=1,\ \ \ve_{1,j}(\uv ),\, \eta_{1,j}(\uv )\in O_S^*,
\end{equation}
where $\ve_{1,j}(\uv ),\eta_{1,j}(\uv )$ are the images of $\ve_1,\eta_1$ under 
$\irS_{\uv ,j}$.

We estimate from above the upper bound \eqref{6.3} from Proposition \ref{le:6.1}.
By assumption, $f$ has degree at most $d_1$ and logarithmic height at most $h_1$,
hence
\begin{equation}\label{6.x1}
|f(\uv )|\leq d_1^qe^{h_1}\max(1,|\uv |)^{d_1}=: R(\uv ).
\end{equation}
Since the degree of $L$ is $d_L\leq D$,
the cardinality $s$ of $S$ is at most $s\leq D(1+\omega )$,
where $\omega$ is the number of prime divisors
of $f(\uv )$. Using the inequality from prime number theory, 
$\omega \leq O(\log |f(\uv )|/\log\log |f(\uv)|)$,
we obtain
\begin{equation}\label{6.x2}
s\leq O\Big(\frac{D\log^* R(\uv )}{\log^*\log^*R(\uv)}\Big).
\end{equation}
From this, one easily deduces that
\begin{equation}\label{6.x3}
c_1\leq \exp O(D\log^* R(\uv )).
\end{equation}
Next, we estimate $P$ and $R_S$. By \eqref{6.x1}, we have
\begin{equation}\label{6.x4}
P\leq Q\leq |f(\uv )|^D\leq \exp O(D\log^* R(\uv )).
\end{equation}
To estimate $R_S$, we use \eqref{6.1}. By Lemma \ref{le:5.7} (using $d_0\leq d_1$) we have
\[
|\Delta_L|\leq D^{2D-1} \big( d_1^qe^{h_1}\max (1,|\uv |)^{d_1}\big)^{2D-2}\leq \exp O(D\log^* R(\uv )),
\]
and this easily implies
\[
|\Delta_L|^{1/2}(\log^*\Delta_L)^{D-1}\leq \exp O(D\log^* R(\uv )).
\]
Together with the estimates \eqref{6.x2},\eqref{6.x4} for $s$ and $Q$, this leads to
\begin{equation}\label{6.x5}
R_S\leq\exp O\Big( D\log^* R(\uv )+s\log^*\log^* Q\Big)\leq \exp O(D\log^* R(\uv )).
\end{equation}
Now by collecting \eqref{6.x3}--\eqref{6.x5}, we infer that the right-hand side
of \eqref{6.3} is bounded above by $\exp O(D\log^* R(\uv ))$. 
So applying Proposition \ref{le:6.1} to \eqref{6.y} gives
\begin{equation}\label{6.x6}
h(\ve_{1,j}(\uv )),\, h(\eta_{1,j}(\uv ))\leq \exp O(D\log^* R(\uv )).
\end{equation}

We apply Lemma \ref{le:5.6} with $N:= 4D^2(q+d_1+1)^2$. 
From the already established \eqref{3.13} it follows that
$\Deg \ve_1,\, \Deg \eta_1\leq N$. Further, since $d_1\geq d_0$ we have
$N\geq 2Dd_0+2(d_1+1)(q+1)$. So indeed, Lemma \ref{le:5.6} is applicable
with this value of $N$. It follows that the set
$\SS :=\{ \uv\in\Zz^q :\, |\uv |\leq N,\, \HH (\uv )\not= 0\}$ is not empty.
Further, for $\uv\in\SS$, $j=1\kdots D$, we have
\begin{eqnarray*}
h(\ve_{1,j}(\uv ))&\leq& \exp O(Dq\log d_1+Dh_1+Dd_1\log^* N)
\\[0.15cm]
&\leq& \exp O(N^{1/2}\log^* N +Dh_1),
\end{eqnarray*}
and so by Lemma \ref{le:5.6},
\[
\Height (\ve_1)\leq \exp O(N^{1/2}\log^* N +Dh_1).
\]
For $\Height (\eta_1 )$ we obtain the same upper bound.
This easily implies \eqref{3.14} in the case $q>0$. 

Now assume $q=0$.
In this case,
$K_0=\Qq$, $A_0=\Zz$ and $B=\Zz [f^{-1},y]$ where $y$ is an algebraic integer 
with minimal polynomial $\FF = X^D+\FF_1X^{D-1}+\cdots +\FF_D\in\Zz [X]$ over $\Qq$,
and $f$ is a non-zero rational integer.
By assumption, $\log |f|\leq h_1$, $\log |\FF_i|\leq h_1$
for $i=1\kdots D$.
Denote by $y_1\kdots y_D$ the conjugates of $y$, and let $L=\Qq (y_j)$ for some $j$.
By a similar argument as in the proof of Lemma \ref{le:5.7}, we have 
$|\Delta_L|\leq D^{2D-1}e^{(2D-2)h_1}$.
The isomorphism given by $y\mapsto y_j$ maps $K$ to $L$ and $B$ to $O_S$, where
$S$ consists of the infinite places of $L$ and of the prime ideals of $O_L$ that divide $f$. 
The estimates \eqref{6.x1}--\eqref{6.x5} remain valid if we replace
$R(\uv )$ by $e^{h_1}$. Hence for any solution 
$\ve_1,\eta_1$ of \eqref{3.12},
\[
h(\ve_{1,j}),\, h(\eta_{1,j})\leq \exp O(Dh_1),
\]
where $\ve_{1,j}$,$\eta_{1,j}$ are the $j$-th conjugates of $\ve_1,\eta_1$, respectively.
Now an application of Lemma \ref{le:5.2}
with $g=\FF$, $m=D$, $\beta_j=\ve_{1,j}$ gives
\[
\Height (\ve_1 )\leq \exp O(Dh_1).
\]
Again we derive the same upper bound for $\Height (\eta_1)$,
and deduce \eqref{3.14}.
This completes the proof of Proposition \ref{pr:3.6}.
\end{proof}

\section{Proof of Theorem \ref{th:1.3}}\label{7} 

We start with some results on multiplicative (in)dependence.

\begin{lemma}\label{le:7.1}
Let $L$ be an algebraic number field of degree $d$, and $\gamma_0\kdots \gamma_s$ non-zero
elements of $L$ such that $\gamma_0\kdots\gamma_s$ are multiplicatively dependent,
but any $s$ elements among $\gamma_0\kdots \gamma_s$ are multiplicatively independent.
Then there are non-zero integers $k_0\kdots k_s$ such that
\begin{eqnarray*}
&&\gamma_0^{k_0}\cdots \gamma_s^{k_s}=1,
\\[0.15cm]
&&|k_i|\leq 58(s!e^s/s^s)d^{s+1}(\log d)h(\gamma_0)\cdots h(\gamma_s)/h(\gamma_i)\ \ 
\mbox{for } i=0\kdots s.
\end{eqnarray*}
\end{lemma}

\begin{proof} This is Corollary 3.2 of Loher and Masser \cite[2004]{LoMa04}.
They attribute this result to Yu Kunrui. Another result of this type was obtained earlier by
Loxton and van der Poorten \cite[1983]{LvdP83}.
\end{proof}

We prove a generalization for arbitrary finitely generated domains.
As before, let $A=\Zz [z_1\kdots z_r]\supseteq \Zz$ be a domain,
and suppose that the ideal $I$ of polynomials $f\in\Zz [X_1\kdots X_r]$
with $f(z_1\kdots z_r)=0$ is generated by $f_1\kdots f_m$.
Let $K$ be the quotient field of $A$. Let $\gamma_0\kdots\gamma_s$ be non-zero
elements of $K$, and for $i=1\kdots s$, let $(g_{i1},g_{i2})$ be a pair of representatives for $\gamma_i$,
i.e., elements of $\Zz [X_1\kdots X_r]$
such that 
\[
\gamma_i =\frac{g_{i1}(z_1\kdots z_r)}{g_{i2}(z_1\kdots z_r)}.
\]

\begin{lemma}\label{le:7.2}
Assume that $\gamma_0\kdots\gamma_s$ are multiplicatively dependent.
Further, assume that $f_1\kdots f_m$ and $g_{i1},g_{i2}$ ($i=0\kdots s$)
have degrees at most $d$ and logarithmic heights at most $h$, where $d\geq 1$, $h\geq 1$.
Then there are 
integers $k_0\kdots k_s$, not all equal to $0$, such that
\begin{eqnarray}
\label{7.2}
&& \gamma_0^{k_0}\cdots \gamma_s^{k_s}=1,
\\[0.15cm]
\label{7.3}
&& |k_i|\leq (2d)^{\exp O(r+s)}(h+1)^s\ \ \mbox{for } i=0\kdots s.
\end{eqnarray}
\end{lemma}

\begin{proof}
We assume without loss of generality that any $s$ numbers among
$\gamma_0\kdots\gamma_s$ are multiplicatively independent
(if this is not the case, take a
minimal multiplicatively dependent subset of $\{ \gamma_0\kdots \gamma_s\}$
and proceed further with this subset).
We first assume that $q>0$. 
We use an argument of van der Poorten and Schlickewei \cite[1991]{vdPSchl91}.
We keep the notation and assumptions from Sections \ref{3}--\ref{5}.
In particular, we assume that $z_1\kdots z_q$ is a transcendence basis of $K$,
and rename $z_{q+1}\kdots z_r$ as $y_1\kdots y_t$, respectively.
For brevity, we have included the case $t=0$ as well in our proof.
But it should be possible to prove in this case a sharper result 
by means of a more elementary method.
In the case $t>0$, $y$ and $\FF =X^D+\FF_1X^{D-1}+\cdots +\FF_D$ will be as
in Corollary \ref{co:3.3}. In the case $t=0$ we take $m=1$, $f_1=0$, $d=h=1$,
$y=1$, $\FF =X-1$, $D=1$.
We construct a specialization such that among the images of $\gamma_0\kdots\gamma_s$
no $s$ elements are multiplicatively dependent, and then apply Lemma \ref{le:7.1}.
 
Let $V\geq 2d$ be a positive integer. Later we shall make our choice of $V$ more precise.
Let 
\begin{eqnarray}\label{7.3a}
&&\VV :=\{ \vv =(v_0\kdots v_s)\in\Zz^{s+1}\setminus\{ \nullv\}:
\\[0.1cm]
\nonumber
&&\hspace*{2cm}
|v_i|\leq V\ \mbox{for } i=0\kdots s, \mbox{and with $v_i=0$ for some $i$}\}.
\end{eqnarray}
Then
\[
\gamma_{\vv}:= \Big(\prod_{i=0}^s\gamma_i^{v_i}\Big)-1\ \ (\vv\in\VV )
\]
are non-zero elements of $K$.  
It is not difficult to show that for $\vv\in\VV$,
$\gamma_{\vv}$ has a pair of representatives $(g_{1,\vv},g_{2,\vv})$ such that
\[
\deg g_{1,\vv},\, \deg g_{2,\vv}\leq sdV.
\]
In the case $t>0$, there exists by Lemma \ref{le:3.4a} a non-zero $f\in A_0$ such that
\[
A\subseteq B:= A_0[y,f^{-1}],\ \ \gamma_{\vv}\in B^*\ \mbox{for $\vv\in\VV$}
\]
and 
\[
\deg f\leq V^{s+1}(2sdV)^{\expr}\leq V^{\exp O(r+s)}.
\]
In the case $t=0$ this holds true as well, 
with $y=1$ and  $f=\prod_{\vv\in\VV} ((g_{1,\vv}\cdot g_{2,\vv})$.
We apply the theory on specializations explained in Section \ref{5} with this $f$.
We put $\HH :=\Delta_{\FF}\FF_Df$, where
$\Delta_{\FF}$ is the discriminant of $\FF$. 
Using Corollary \ref{co:3.3} and inserting the bound $D\leq d^t$ from Lemma \ref{le:3.1}
we get for $t>0$,
\begin{equation}\label{7.4}
\left\{\begin{array}{l}
d_0:=\max (\deg f_1\kdots\deg f_m,\deg\FF_1\kdots\deg\FF_D)\leq (2d)^{\expr} ,
\\[0.15cm]
h_0:=\max \big(h(f_1)\kdots h(f_m),h(\FF_1)\kdots h(\FF_D)\,\big)\leq (2d)^{\expr}(h+1)\, ;
\end{array}\right.
\end{equation}
with the provision $\deg 0 = h(0)= -\infty$ this is true also if $t=0$.
Combining this with Lemma \ref{le:3.4}, we obtain
\[
\deg\HH \leq (2D-1)d_0 +\deg f \leq  V^{\exp O(r+s)}.
\]
By Lemma \ref{le:5.3} there exists $\uv\in\Zz^q$ with
\begin{equation}\label{7.5}
\HH (\uv )\not =0,\ \ |\uv |\leq V^{\exp O(r+s)}.
\end{equation}
We proceed further with this $\uv$.

As we have seen before, $\gamma_{\vv}\in B^*$ for $\vv\in\VV$.
By our choice of $\uv$, 
there are $D$ distinct specialization maps $\irS_{\uv ,j}$ ($j=1\kdots D)$
from $B$ to $\OQq$. We fix one of these specializations, say $\irS_{\uv}$.
Given $\alpha\in B$, we write $\alpha (\uv )$ for $\irS_{\uv}(\alpha )$.
As the elements $\gamma_{\vv}$ are all units in $B$, their images under $\irS_{\uv}$ are non-zero.
So we have
\begin{equation}\label{7.5a}
\prod_{i=0}^s \gamma_i(\uv )^{v_i}\not=1\ \ \mbox{for } \vv\in\VV,
\end{equation}
where $\VV$ is defined by \eqref{7.3a}.

We use Lemma \ref{le:5.5} to estimate the heights $h(\gamma_i(\uv ))$ for $i=0\kdots s$.
Recall that by Lemma \ref{le:3.4} we have
\[ 
\Deg \gamma_i\leq (2d)^{\expr},\ \  \Height (\gamma_i)\leq (2d)^{\expr}(h+1)
\]
for $i=0\kdots s$.
By inserting these bounds, together with the bound $D\leq d^t$ from Lemma \ref{le:3.1},
those for $d_0,h_0$ from \eqref{7.4} and that for $\uv$
from \eqref{7.5} into the bound from Lemma \ref{le:5.5}, we obtain for $i=0\kdots s$,
\begin{eqnarray}\label{7.6}
h(\gamma_i(\uv ))&\leq& (2d)^{\expr}(1+h+\log\max (1,|\uv |))
\\[0.15cm]
\nonumber
&\leq& (2d)^{\exp O(r+s)}(1+h+\log V).
\end{eqnarray}

Assume that among $\gamma_0(\uv )\kdots \gamma_s(\uv )$ there are $s$ numbers
which are multiplicatively dependent. 
By Lemma \ref{le:7.1} there are integers $k_0\kdots k_s$, 
at least one of which is non-zero and at least one of which is $0$,
such that
\begin{eqnarray*}
&&\prod_{i=0}^s \gamma_i(\uv )^{k_i}=0,
\\
&&|k_i|\leq (2d)^{\exp O(r+s)}(1+h+\log V)^{s-1}\ \ \mbox{for } i=0\kdots s.
\end{eqnarray*}
Now for 
\begin{equation}\label{7.7}
V=(2d)^{\exp O(r+s)}(h+1)^{s-1}
\end{equation}
(with a sufficiently large constant in the O-symbol),
the upper bound for the numbers $|k_i|$ is smaller than $V$. But this would imply that
$\prod_{i=0}^s \gamma_i(\uv )^{v_i}=1$ for some $\vv\in\VV$, contrary to \eqref{7.5a}.
Thus we conclude that with the choice \eqref{7.7} for $V$, there exists $\uv\in\Zz^q$ with \eqref{7.5},
such that any $s$ numbers among $\gamma_0(\uv )\kdots\gamma_s (\uv )$ are multiplicatively
independent. Of course, the numbers $\gamma_0(\uv )\kdots\gamma_s(\uv )$ are multiplicatively dependent,
since they are the images under $\irS_{\uv}$ of $\gamma_0\kdots\gamma_s$ which are
multiplicatively dependent. Substituting \eqref{7.7} into \eqref{7.6} we obtain
\begin{equation}\label{7.7a}
h(\gamma_i(\uv ))\leq (2d)^{\exp O(r+s)}(h+1)\ \ \mbox{for } i=0\kdots s.
\end{equation}
Now Lemma \ref{le:7.1} implies that there are non-zero integers $k_0\kdots k_s$ such that
\begin{eqnarray}\label{7.8}
&&\prod_{i=0}^s \gamma_i(\uv )^{k_i}=1,
\\[0.15cm]
\label{7.9}
&&|k_i|\leq (2d)^{\exp O(r+s)}(h+1)^s\ \ \mbox{for } i=0\kdots s.
\end{eqnarray}

Our assumption on $\gamma_0\kdots \gamma_s$ implies that there are non-zero integers
$l_0\kdots l_s$ such that $\prod_{i=0}^s \gamma_i^{l_i}=1$. Hence
$\prod_{i=0}^s \gamma_i(\uv )^{l_i}=1$. Together with \eqref{7.8} this implies
\[
\prod_{i=1}^s \gamma_i(\uv )^{l_0k_i-l_ik_0}=1.
\]
But $\gamma_1(\uv )\kdots \gamma_s(\uv )$ are multiplicatively independent, hence
$l_0k_i-l_ik_0=0$ for $i=1\kdots s$. That is,
\[
l_0(k_0\kdots k_s)=k_0(l_0\kdots l_s).
\]
It follows that 
\[
\prod_{i=0}^s \gamma_i^{k_i}=\rho
\]
for some root of unity $\rho$. But $\irS_{\uv}(\rho )=1$ and it is conjugate to $\rho$.
Hence $\rho =1$. So in fact we have
$\prod_{i=0}^s \gamma_i^{k_i}=1$ with non-zero integers $k_i$ satisfying \eqref{7.9}.
This proves our Lemma, but under the assumption $q>0$. If $q=0$ then a much
simpler argument, without specializations, 
gives $h(\gamma_i)\leq (2d)^{\exp O(r+s)}(h+1)$ for $i=0\kdots s$ instead of \eqref{7.7a}.
Then the proof is finished in the same way as in the case $q>0$.
\end{proof}

\begin{corollary}\label{co:7.3}
Let $\gamma_0,\gamma_1\kdots\gamma_s\in K^*$,
and suppose that 
$\gamma_1\kdots\gamma_s$ are multiplicatively independent
and 
\[
\gamma_0=\gamma_1^{k_1}\cdots\gamma_s^{k_s}
\]
for certain integers $k_1\kdots k_s$. Then
\[
|k_i|\leq (2d)^{\exp O(r+s)}(h+1)^s\ \ \ \mbox{for $i=1\kdots s$.}
\]
\end{corollary}

\begin{proof}
By Lemma \ref{le:7.2}, and by the multiplicative independence of $\gamma_1\kdots\gamma_s$,
there are integers $l_0\kdots l_m$ such that
\begin{eqnarray*}
&&\prod_{i=0}^m \gamma_i^{l_i}=1,
\\[0.15cm]
&&l_0\not= 0,\ \ |l_i|\leq (2d)^{\exp O(r+s)}(h+1)^s\ \mbox{for $i=0\kdots s$.}
\end{eqnarray*}
Now clearly, we have also
\[
\prod_{i=1}^s \gamma_i^{l_0k_i-l_i}=1,
\]
hence $l_0k_i-l_i=0$ for $i=1\kdots s$. It follows that $|k_i|=|l_i/l_0|\leq (2d)^{\exp O(r+s)}(h+1)^s$
for $i=1\kdots s$. This implies our Corollary.
\end{proof}

\begin{proof}[Proof of Theorem \ref{th:1.3}]
We keep the notation and assumptions from the statement of Theorem \ref{th:1.3}.
Define the domain 
\[
\widetilde{A}:= A[\gamma_1,\gamma_1^{-1}\kdots \gamma_s,\gamma_s^{-1}].
\]
Then
\[
\widetilde{A}\cong \Zz [X_1\kdots X_r,X_{r+1}\kdots X_{r+2s}]/\widetilde{I}
\]
with
\begin{eqnarray*}
\widetilde{I}&=&\Big(f_1\kdots f_m,g_{12}X_{r+1}-g_{11},g_{11}X_{r+2}-g_{12},\ldots
\\
&&\hspace*{2cm} \ldots , g_{s2}X_{r+2s-1}-g_{s1},g_{s1}X_{r+2s}-g_{s2}\Big).
\end{eqnarray*}
Let $(v_1\kdots w_s)$ be a solution of \eqref{1.4}, and put $\ve :=\prod_{i=1}^s \gamma_i^{v_i}$,
$\eta := \prod_{i=1}^s \gamma_i^{w_i}$. Then
\[
a\ve +b\eta =c,\ \ \ \ve ,\eta\in \widetilde{A}^* .
\]
By Theorem \ref{th:1.1}, $\ve$ has a representative $\widetilde{\ve}\in\Zz [X_1\kdots X_{r+2s}]$ of degree
and logarithmic height both bounded above by
\[
\exp\Big( (2d)^{\exp O(r+s)}(h+1)\Big).
\]
Now Corollary \ref{co:7.3} implies
\[
|v_i|\leq \exp\Big( (2d)^{\exp O(r+s)}(h+1)\Big )\ \ \mbox{for } i=1\kdots s.
\]
For $|w_i|$ ($i=1\kdots s$) we derive a similar upper bound.
This completes the proof of Theorem \ref{th:1.3}.
\end{proof}

\end{document}